\documentclass[a4paper,11pt,reqno]{amsart}
\usepackage{amssymb,amsmath,mathrsfs}

\usepackage{times}

\usepackage[utf8]{inputenc}
\usepackage{comment}
\usepackage[english]{babel}
\usepackage{amsmath, mathtools}
\usepackage{amssymb}
\usepackage{amsfonts}
\usepackage{amsthm}
\usepackage{bbm}
\usepackage{enumitem}
\usepackage{xcolor}
\usepackage[colorlinks=false]{hyperref}

\def\d{\mathrm{d}}

 \providecommand{\Xint}[1]{\mathchoice
    {\XXint\displaystyle\textstyle{#1}}%
    {\XXint\textstyle\scriptstyle{#1}}%
    {\XXint\scriptstyle\scriptscriptstyle{#1}}%
    {\XXint\scriptscriptstyle\scriptscriptstyle{#1}}%
    \!\int}
  \providecommand{\XXint}[3]{{\setbox0=\hbox{$#1{#2#3}{\int}$}
      \vcenter{\hbox{$#2#3$}}\kern-.5\wd0}}
  
  \providecommand{\dashint}{\mathop{\Xint-}}

\def\ristretto{\lfloor}

\usepackage{ esint }

\newtheorem{theorem}{Theorem}[section]
\newtheorem{corollary}[theorem]{Corollary}

\newtheorem{lemma}[theorem]{Lemma}
\newtheorem{proposition}[theorem]{Proposition}

\theoremstyle{definition}
\newtheorem{definition}[theorem]{Definition}
\newtheorem{remark}[theorem]{Remark}

\numberwithin{equation}{section}

\newcommand{\inc}[1]{\hyperref[inc]{{\normalfont(inc){\ensuremath{_{#1}}}}}}
\newcommand{\dec}[1]{\hyperref[dec]{{\normalfont(dec){\ensuremath{_{#1}}}}}}
\newcommand{\azero}{\hyperref[azero]{{\normalfont(A0)}}}
\newcommand{\vauno}{\hyperref[va1]{{\normalfont(VA1)}}}

\newcommand{\R}{\mathbb{R}}
\newcommand{\N}{\mathbb{N}}

\newcommand{\e}{\varepsilon}

\newcommand{\ssubset}{\subset\joinrel\subset}

\newcommand{\Rd}{{\R}^d}

\newcommand{\Ld}{{\mathcal{L}}^d}

\newcommand{\res}{\mathop{\hbox{\vrule height 7pt width .5pt depth 0pt
\vrule height .5pt width 6pt depth 0pt}}\nolimits}

\def\Xint#1{\mathchoice 
  {\XXint\displaystyle\textstyle{#1}}%
  {\XXint\textstyle\scriptstyle{#1}}%
  {\XXint\scriptstyle\scriptscriptstyle{#1}}%
  {\XXint\scriptscriptstyle\scriptscriptstyle{#1}}%
  \!\int} 
\def\XXint#1#2#3{{\setbox0=\hbox{$#1{#2#3}{\int}$} 
  \vcenter{\hbox{$#2#3$}}\kern-.5\wd0}} 
\def\-int{\Xint -}

\newcommand{\Phiw}{\Phi_{\rm{w}}}
\newcommand{\Phic}{\Phi_{\rm{c}}}

\usepackage{latexsym,amsfonts}
\usepackage{mathrsfs}
\usepackage{centernot}
\numberwithin{equation}{section}

\begin{document}  

\date{\today}

\title[Free-discontinuity problems  with non-standard growth]{Strong existence for  free-discontinuity problems  with non-standard growth}

\author{Chiara Leone}
\address[Chiara Leone]{Department of Mathematics and Applications ``R. Caccioppoli'', University of Naples Federico II, Via Cintia, Monte S. Angelo, 80126 Naples, Italy}
\email{chiara.leone@unina.it}
 
\author{Giovanni Scilla}
\address[Giovanni Scilla]{Department of Mathematics and Applications ``R. Caccioppoli'', University of Naples Federico II, Via Cintia, Monte S. Angelo, 80126 Naples, Italy}
\email[Giovanni Scilla]{giovanni.scilla@unina.it}

\author{Francesco Solombrino}
\address[Francesco Solombrino]{Department of Mathematics and Applications ``R. Caccioppoli'', University of Naples Federico II, Via Cintia, Monte S. Angelo, 80126 Naples, Italy}
\email{francesco.solombrino@unina.it}

\author{Anna Verde}
\address[Anna Verde]{Department of Mathematics and Applications ``R. Caccioppoli'', University of Naples Federico II, Via Cintia, Monte S. Angelo, 80126 Naples, Italy}
\email{anna.verde@unina.it}

\subjclass[2020]{49J45, 46E30, 35B65}

\keywords{free-discontinuity problems, non-standard growth, regularity, minimizers}

\begin{abstract}
{An Ahlfors-type regularity result for free-discontinuity energies defined on the space $SBV^{\varphi}$ of special functions of bounded variation with $\varphi$-growth, where $\varphi$ is a generalized Orlicz function, is proved. Our analysis expands on the regularity theory for minimizers of a class of free-discontinuity problems in the non-standard growth case. } 
\end{abstract}

\maketitle

\tableofcontents

\section{Introduction}

This paper is concerned with existence of strong minimizers for a model functional of the form
\begin{equation}\label{modello-strong}
\int_\Omega\varphi(x,|\nabla u|)\,dx+ \mathcal{H}^{d-1} (K),
\end{equation}
to be minimized on pairs $(u, K)$ of smooth functions $u$ outside of a closed discontinuity set $K$, whose Hausdorff measure $\mathcal{H}^{d-1} (K)$ is penalised. The prototypical example (for $\varphi(x, \xi)=|\xi|^2$) was introduced by Mumford and Shah \cite{MS} in the context of image segmentation, and later used as well for describing failure phenomena such as fracture and damage in elastic materials after the seminal paper by Francfort and Marigo \cite{FrM} on the variational reformulation of Griffith's theory of brittle fracture. In our paper we will however not assume a standard $p$-growth assumption for the bulk energy integrand, but rather a non-standard growth in generalized  Orlicz type of spaces (see, e.g.,~\cite{HH}). As relevant examples of bulk energies undergoing non-standard growth we report here the perturbed {\em Orlicz}, the so-called {\em variable exponent}, 
and the {\em double-phase} case
	\begin{equation}
	\label{e:intro1}
	a(x)\varphi(|\xi|),\qquad
	|\xi|^{p(x)}, \qquad\text{and} \qquad |\xi|^{p} + a(x) |\xi|^{q} \qquad \text{for $(x, \xi) \in \R^{d} \times \R^{d}$,}
	\end{equation}
for suitable choices of the $N$-function $\varphi:[0,+\infty)\to[0,+\infty)$, of the exponent function $p \colon \R^{d} \to (1, +\infty)$, of the exponents  $1 < p < q <+\infty$, and of the weight function~$a \colon \R^{d} \to [0, +\infty)$, while a list of relevant examples in the literature can be found, for instance in \cite[Section 4.3]{ars}.  

\noindent In a Sobolev setting, that is when the  term $\mathcal{H}^{d-1} (K)$ does not appear, integral functionals with  non-standard growth first appeared in the works of Zhikov~\cite{zikov, zikov3} for modeling composite materials characterized by a strongly anisotropic behavior, and have attracted an increasing attention in the last decades. A huge amount  of results for these models, their variants as well as borderline cases  have been investigated. Each single case has been studied in a peculiar way, relying on its particular structure. A unified approach to treat several cases of non-standard growth  has been recently  proposed in the monograph \cite{HH}, which can be consulted also for the rich bibliography (see also \cite{HastoOk} and the references therein for the regularity topic).

The coupling of bulk energies of this kind with a free discontinuity term is therefore a natural possibility in the variational approach to the emergence of singularities with codimension $1$ in anisotropic media.

\noindent The usual strategy for proving well-posedness of \eqref{modello-strong} goes through a weak reformulation in the space  $SBV$ of special functions of bounded variation (see Section \ref{s:bv}) of the form
\begin{equation}\label{modello-weak}
\int_\Omega\varphi(x,|\nabla u|)\,dx+ \mathcal{H}^{d-1} (J_u).
\end{equation}
Above, $J_u$ is the set of jump discontinuities of $u$ with normal $\nu_u$, which, exactly like the gradient $\nabla u$, has in general to be understood in an approximate measure-theoretical sense. If $\varphi$ is superlinear, and convex in the second variable, existence of weak minimizers can be recovered by combining  De Giorgi's lower semicontinuity theorem \cite{DeG} with the closure and compactness results in $SBV$ by Ambrosio (see \cite[Theorem~4.7 and 4.8]{Ambrosio-Fusco-Pallara:2000}). Then, in order to get strong existence, the crucial point is to prove that minimizers  of \eqref{modello-weak} have an essentially closed jump set and are smooth outside of it. Practically, this is achieved by proving that the singular set $S_u$, a superset of the crack set $J_u$ (which in a $BV$ setting, differs therefrom only by a $\mathcal{H}^{d-1}$-null set),  is locally {\it Ahlfors-regular}, meaning that it satisfies the uniform density estimate
\[
\mathcal{H}^{d-1}(S_u\cap B_\rho(x_0)) \ge \theta_0 \rho^{d-1}
\]
for $x_0\in S_u$ and sufficiently small balls $B_\rho(x_0)$, with $\theta_0$ independent of $x_0$ and $\rho$.  

For the model case $\varphi(x, \xi)=|\xi|^p$, this result was obtained in the seminal paper \cite{DeGCarLea}.  There, a contradiction-compactness argument shows that, in regions with small crack, the energy in a ball $B_\rho(x_0)$ decays on the order of $\rho^d$ (like a bulk energy), much faster than the surface energy scaling $\rho^{d-1}$. A crucial step to this aim is represented by a Poincar\'e-Wirtinger inequality for $SBV$ functions with small jump set in a ball \cite[Theorem~4.14]{Ambrosio-Fusco-Pallara:2000}, which  is exploited to replace (a suitable rescaling of) optimal  sequences with more regular functions, without excessively increasing the energy. Such an inequality is obtained with respect to the $L^p$-norm, and its proof heavily exploits its homogeneity. The result in \cite{DeGCarLea} has been generalized to  nonhomogenous bulk integrands with $p$-growth in \cite{FMT}, and recently to the variable exponent  setting in \cite{LSSV}  (see also \cite{DiPS} for constrained energy functionals in  the variable exponent setting). Also the vector-valued case of linearly elastic bulk integrands (that is, depending only on the symmetrized part $e(u)$ of the strain $\nabla u$) has been investigated in \cite{Chambolle-Crismale,  Conti-Focardi-Iurlano,  Friedrich-Labourie-Stinson} through  a delicate reformulation in spaces of functions of bounded deformation, where a number of additional technical issues have to be overcome.


\noindent \emph{Description of our results.} In this paper we will deal with integrands $\varphi(x, |\xi|)$ which are convex in the last variable. We do not focus on a single case study, but rather follow a general perspective fixing a fair set of assumptions on the bulk integrand $\varphi$, under which existence of strong minimizers can be recovered. These conditions, summarized in \ref{hpuno}-\ref{hptre}, are borrowed from the regularity theory for variational integrals in generalized Orlicz spaces \cite{HH, HastoOk}, and here specified to the free-discontinuity setting. In particular, the Orlicz setting (when $\varphi'(t)$ is equivalent to $t\varphi''(t)$), the variable exponent setting (under strongly log-H\"older continuity of the variable exponent), and the double-phase case (under suitable H\"older continuity of the function $a(x)$ in \eqref{e:intro1}) are covered by our results (see Section 8 in \cite{HastoOk}, where the authors show that \ref{hpuno}-\ref{hptre} include these special structures).
Notice that \ref{hpuno}-\ref{hpdue} imply in particular that the growth from above and below of the energy is ruled by two possibly different exponents $1<p<q$.  Condition \ref{hptre} is fundamental to our blow-up analysis. Roughly speaking, this condition does not allow for a too degenerate behavior on small balls contained in~$\Omega$ of the functions
\[
\varphi^{+}_{B} (t) := \sup_{x \in B} \, \varphi(x, t) \qquad \varphi^{-}_{B} (t):= \inf_{x \in B} \, \varphi(x, t) \qquad \text{for $t \in [0, +\infty)$ and $B \subseteq \Omega$},
\]
since the first one has to be controlled by the second one, up to  a vanishing error when the size of the ball goes to zero. 

The proof of the crucial density lower bound goes, exactly as in  \cite{DeGCarLea}, through a decay lemma (see Lemma \ref{lem:decay}). One assumes by contradiction that the energy is decaying faster than $\rho^{d-1}$ around a jump point $x_0$. The goal is then to show that a scaled copy of blown-up sequences converges to a  Orlicz-Sobolev minimizer of a variational integral of the type
\[
\int_{B_1} \varphi_\infty(|\nabla u|))\,\mathrm{d}x,
\]
where the function $\varphi_\infty$ does not depend on $x$ and is recovered as locally uniform limit of a scaled version of $\varphi$ acting on a scaled deformation gradient. For the above problem, decay estimates only depending on the dimension and the growth exponents $p$ and $q$ are provided by \cite{DSV,HastoOk} and can be used to get a contradiction.
A crucial point for this procedure is the construction of nearly optimal sequence whose Orlicz norm is controlled, despite the presence of a small jump set, only in terms of the Orlicz norm of the gradient. This can be achieved by means of a Poincar\'e-type inequality recently introduced in \cite{ars}, which is proved under quite different lines than in \cite{DeGCarLea}, relying on some ideas in \cite{CianchiFuscoCrelle, CianchiFuscoBV}. Actually, in our setting a slight refinement thereof is needed, that is summarized in Theorem \ref{thm:poincsbvphi}. It can be obtained with a prompt adaptation of the arguments in \cite[Section 5]{ars}. In  the form we state, it eventually allows one to get equi-integrability estimates for some remainder terms appearing in Lemma \ref{lem:decay}.

Actually, {an application} of the  Poincar\'e-type inequality is not enough to our purposes, as the bulk energy may explicitly depend on the point $x$ in the reference configuration. This is a point where \ref{hptre} comes essentially into play, in order to assure that some remainder terms appearing in the construction can be taken uniformly small.
Once this is established, one can recover existence of strong minimizers in Theorem \ref{thm:main} following in the footsteps of the classical proof  in  \cite{DeGCarLea}, up to the necessary adaptations to the generalized Orlicz setting.

\noindent \emph{Outlook.} We provide a well-posedness result for \eqref{modello-strong} under a generalized Orlicz growth for the bulk integrand, establishing of a general framework for the analysis of the problem, which encompasses diverse relevant case studies under a unified perspective. Along with the usage of   the Poincar\'e-type inequality for $SBV$ functions with Orlicz growth in the gradient, together with assumption \ref{hptre}, these are the main points where our analysis departs from the classical program in \cite{DeGCarLea}. It must be however mentioned that some interesting examples are still not covered by the present analysis and deserve further investigation,  possibly requiring the introduction of new tools. For instance, an $L\log L$-growth on $\varphi$, although superlinear, is excluded by \eqref{inc}. While the regularity properties of minimizers of the corresponding variational integrals in Sobolev spaces is by now well understood (see \cite{DeFM,FM}), a blow-up procedure would likely deliver no significant information in the free-discontinuity setting, as the limit function $\varphi_\infty$ turns out to be linear. A different strategy has therefore to be found in order to deal with such a problem. 

\noindent Let us also remark that the existence theory for weak minimizers can encompass a broader range of situations, in particular in the vectorial case: $\varphi(x, \xi)$  may indeed be a non radial, quasiconvex integrand with $p$-growth (\cite[Theorem~5.29]{Ambrosio-Fusco-Pallara:2000}) or, under suitable assumptions, with a variable exponent or Orlicz growth (\cite{ars, DCLV, LV}).  Also the investigation of general free discontinuity functionals with bulk energies having  a mixed $(p,q)$-growth condition (which is implied by, but not equivalent to   \ref{hpuno}-\ref{hpdue}), relevant for the modeling of determinant constraints, is not yet well understood. In this case, even the well-posedness of \eqref{modello-weak} is actually not clear, by lack of suitable lower semincontinuity results.

\section{Basic notation and preliminaries}\label{sec: prel}

 We start with some basic notation.   Let $\Omega \subset \R^d$  be  open and bounded.
 For every $x\in \Rd$ and $r>0$ we indicate by $B_r(x) \subset \Rd$ the open ball with center $x$ and radius $r$. If $x=0$, we will often use the shorthand $B_r$.  For $x$, $y\in \Rd$, we use the notation $x\cdot y$ for the scalar product and $|x|$ for the  Euclidean  norm.   
The $m$-dimensional Lebesgue measure of the unit ball in $\R^m$ is indicated by $\kappa_m$ for every $m \in \N$.   We denote by $\Ld$ and $\mathcal{H}^k$ the $d$-dimensional Lebesgue measure and the $k$-dimensional Hausdorff measure, respectively.  
The closure of $A$ is denoted by $\overline{A}$. The diameter of $A$ is indicated by ${\rm diam}(A)$.  
We write $\chi_A$ for the  characteristic  function of any $A\subset  \R^d$, which is 1 on $A$ and 0 otherwise.  

Given two functions $f,g:[0,+\infty)\to\mathbb{R}$, we write $f\sim g$, and we say that $f$ and $g$ are equivalent, if there exist constants $c_{1}, c_{2} >0$ such that $c_{1}g(t) \leq f(t) \leq c_{2}g(t)$ for any $t\ge 0$. Similarly the symbol $\lesssim$ stands for $\le$ up to a constant. {$L^0(\Omega)$ denotes the set of the measurable functions on $\Omega$.}

\subsection{Generalized $\Phi$-functions and Orlicz spaces}\label{sec:genorl}

We introduce some basic definitions and useful facts about generalized $\Phi$-functions and Orlicz spaces. We will restrict to only considering concepts we will use. We refer the reader to \cite{HH} for a comprehensive treatment of the topic. 

{\begin{definition}\label{weakphi}
Let $\varphi: [0,+\infty)\to [0, +\infty]$ be increasing with $\varphi(0) = 0$, $\lim_{t\to 0^+} \varphi(t) = 0$ and $\lim_{t\to+\infty}\varphi(t) = +\infty$. Such $\varphi$ is called a
\begin{enumerate}
\item[(i)] \emph{weak $\Phi$-function} if $\frac{\varphi(t)}{t}$ is \emph{almost increasing}, meaning that there exists $L\ge 1$ such that $\frac{\varphi(t)}{t}\le L\frac{\varphi(s)}{s}$ for $0<t\le s$.
\item[(ii)] \emph{convex $\Phi$-function} if $\varphi$ is left-continuous and convex. 
\end{enumerate}
\end{definition}
By virtue of Remark \ref{rem:varinc}, each convex $\Phi$-function is a weak $\Phi$-function. If $\varphi$ is a convex $\Phi$-function, then there exists $\varphi'$ the \emph{right derivative} of $\varphi$, which is non-decreasing and right-continuous, and such that
\begin{equation*}
\varphi(t)=\int_0^t \varphi'(s)\,\mathrm{d}s \,.
\end{equation*}
A special subclass of convex $\Phi$-functions is represented by the so called ``nice Young functions'', also known as $N$-functions (see, e.g., \cite[Ch.I]{KR}).}

{\begin{definition}\label{Nfunc}
A function $\varphi: [0, \infty)\rightarrow [0, \infty)$ is said to be an \emph{$N$-function} if it admits the representation
\begin{equation*}
\varphi(t)=\int_0^t a(\tau)\,\mathrm{d}\tau
\end{equation*}
where $a(s)$ is right-continuous, non-decreasing for $s>0$, $a(s)>0$ for $s>0$ and satisfies the conditions
\begin{equation}
a(0)=0\,,\quad \lim_{s\to+\infty} a(s) =+\infty\,.
\label{eq:asympright}
\end{equation}
\end{definition}
The function $a(t)$ is nothing else than the right-derivative of $\varphi(t)$. 
As a straightforward consequence of the definition, we have that an $N$-function $\varphi$ is continuous, $\varphi(0)=0$ and $\varphi$ is increasing. Moreover, $\varphi$ is a convex function, and, in view of Remark \ref{rem:varinc}, it satisfies \inc{1}.}
{ Conditions \eqref{eq:asympright} imply
\begin{equation}
\lim_{t\to0^+} \frac{\varphi(t)}{t}=0\,,\quad \lim_{t\to +\infty} \frac{\varphi(t)}{t}=+\infty\,.
\label{eq:asympright2}
\end{equation}
It can be shown that an equivalent definition of $N$-function is the following: a continuous convex function $\varphi$ is called an $N$-function if it satisfies \eqref{eq:asympright2}. }

{For our purposes, we need functions $\varphi$ to depend also on the spatial variable $x$. }
{\begin{definition}\label{generalizedfunct}
Let $\varphi:\Omega\times[0,\infty)\to [0,\infty]$. We call 
$\varphi$ a \textit{generalized} weak $\Phi$-function (resp., convex $\Phi$-function, $N$-function) if
\begin{enumerate}
\item[(1)] $x\mapsto \varphi(x,|f(x)|)$ is measurable for every $f\in L^0(\Omega)$; 
\item[(2)] $t\mapsto \varphi(x,t)$ is a weak $\Phi$-function (resp., a convex $\Phi$-function, an $N$-function) for every $x\in\Omega$. 
\end{enumerate}
We write $\varphi\in\Phiw(\Omega)$, $\varphi\in\Phic(\Omega)$ and $\varphi\in N(\Omega)$, respectively. If $\varphi$ does not depend on $x$, we will adopt the shorthands $\varphi\in\Phiw$, $\varphi\in\Phic$ and $\varphi\in N$, respectively. For the right-derivative of a generalized convex $\Phi$-function, we will use the notation $\varphi_t$ in place of $\varphi'$.
\end{definition} }

{For a bounded function $\varphi:\Omega\times[0,+\infty)\to[0,+\infty)$ and a ball $B_r(x_0)\subset\Omega$ we define, for every $t\ge 0$,
\begin{equation}\label{phimeno}
\varphi_{r,x_0}^-(t):=\inf_{x\in B_r(x_0)}\varphi(x,t)\quad\hbox{ and }\quad \varphi_{r,x_0}^+(t) :=  \sup_{x\in B_r(x_0)}\varphi(x,t).
\end{equation}}

{Following the terminology of \cite{HH}, we give the following definitions. The first three ones concern with the regularity of $\varphi$ with respect to the $t$- variable, while the last one is a continuity assumption with respect to the spatial variable $x$. 
\begin{definition}
Let $p,q>0$. A function $\varphi:\Omega\times[0,+\infty)\to[0,+\infty)$ satisfies
\begin{itemize}
\item[\normalfont(inc)$_p$]\label{inc} 
 if $t\in(0,+\infty)\mapsto\frac{\varphi(x,t)}{t^p}$ is increasing for every $x\in\Omega$
\item[\normalfont(dec)$_q$]\label{dec} if $t\in (0,+\infty)\mapsto\frac{\varphi(x,t)}{t^q}$ is decreasing for every $x\in\Omega$
\item[\normalfont(A0)] \label{azero} if there exists $L\ge 1$ such that $\frac{1}{L}\le \varphi (x,1)\le L$ for every $x\in\Omega$
\item[\normalfont(VA1)] \label{va1} if there exists an increasing continuous function $\omega:[0,+\infty)\to [0,1]$ with $\omega(0)=0$ such that, for any ball $B_r(x_0)\subset\Omega$,
\[
\varphi^+_{r,x_0}(t)\le (1+\omega(r))\varphi^-_{r,x_0}(t), \quad\forall t>0 \,\, \mbox{ such that } \,\,\varphi^-_{r,x_0}(t)\in \left [\omega(r), \frac{1}{\mathcal{L}^d(B_r(x_0))} \right].
\]
\end{itemize}
\end{definition} }

{\begin{remark}
If $\varphi$ satisfies \inc{p} (resp., \dec{q}) for some $p>0$ (resp., $q>0$), then so do $\varphi^+_{r,x_0}$ and $\varphi^-_{r,x_0}$ for any $B_{r}(x_0)\subset\Omega$.
\label{rem:scaledvarphi}
\end{remark}}

{\begin{remark} \label{rem:varinc}
If $\varphi:[0,+\infty)\to[0,+\infty)$ is convex and $\varphi(x,0)=0$ for every $x\in\Omega$, then $\varphi$ satisfies \inc{1}. If $\varphi$ satisfies \inc{p_1}, then it satisfies \inc{p_2} for every $0<p_2\leq p_1$. If $\varphi$ satisfies \dec{q_1}, then it satisfies \dec{q_2} for every $q_2\geq q_1$.
\end{remark}}

Next simple results {can be found in \cite[Section 3]{HastoOk}.}

{\begin{proposition}\label{prop:properties}
Let $1<p\le q<+\infty$ and $\varphi\in \Phic(\Omega)$ with right derivative $\varphi_t$. Assume that $\varphi_t$ satisfies \inc{p-1} and \dec{q-1}. Then
\begin{enumerate}
\item[$(i)$] $\varphi$ satisfies \inc{p} and \dec{q}, and the following estimate hold:
\begin{equation}\label{cons2}
\varphi(x,s)\min\{t^p,t^q\}\le\varphi(x,ts)\le\max\{t^p,t^q\}\varphi(x,s),\quad \forall x\in\Omega, \,\, \forall s,t\in[0,+\infty).
\end{equation}
\item[(ii)] $\varphi(x,t)$ and $t\varphi_t(x,t)$ are equivalent, in the sense that \begin{equation}\label{cons1}
p\,\varphi(x,t)\le t\,\varphi_t(x,t)\le q\,\varphi(x,t),\quad \forall (x,t)\in\Omega\times[0,+\infty);
\end{equation}
\item[(iii)] if, in addition, $\varphi_t$ complies with \azero{}, then also $\varphi$ does with constants depending on $L,p,q$. More precisely, 
\begin{equation}\label{v0phi}
\frac{1}{Lq}\le\varphi(x,1)\le \frac{L}{p}, \ \ \ \forall x\in\Omega.
\end{equation}
\end{enumerate}
If, in addition, $\varphi(x, \cdot )\in C^1([0,+\infty))$ for every $x\in\Omega$, then $\varphi\in N(\Omega)$.
\end{proposition} 
}

{
For $\varphi\in\Phiw(\Omega)$, the \textit{generalized Orlicz space} is defined by 
\[
L^{\varphi}(\Omega):=\big\{f\in L^0(\Omega):\|f\|_{L^\varphi(\Omega)}<\infty\big\}
\] 
with the (Luxemburg) norm 
\[
\|f\|_{L^\varphi(\Omega)}:=\inf\bigg\{\lambda >0: \varrho_{\varphi}\Big(\frac{f}{\lambda}\Big)\leq 1\bigg\},
\ \ \text{where}\ \ \varrho_{\varphi}(f):=\int_\Omega\varphi(x,|f(x)|)\,\mathrm{d}x.
\]
We denote by $W^{1,\varphi}(\Omega)$ the set of $f\in L^{\varphi}(\Omega)$ satisfying that $\partial_1f,\dots,\partial_df \in L^{\varphi}(\Omega)$, where $\partial_if$ is the weak derivative of $f$ in the $x_i$-direction, with the norm $\|f\|_{W^{1,\varphi}(\Omega)}:=\|f\|_{L^\varphi(\Omega)}+\sum_i\|\partial_if\|_{L^\varphi(\Omega)}$. Note that if $\varphi$ satisfies \dec{q} for some $q\ge 1$, then $f\in L^\varphi(\Omega)$ if and only if $\varrho_\varphi(f)<\infty$, and if $\varphi$ satisfies \azero{}, \inc{p} and \dec{q} for some $1<p\leq q$, then $L^\varphi(\Omega)$ and $W^{1,\varphi}(\Omega)$ are reflexive Banach spaces. In addition we denote by $W^{1,\varphi}_0(\Omega)$ the closure of $C^\infty_0(\Omega)$ in $W^{1,\varphi}(\Omega)$.}


%


\noindent We also need the following definitions {and results about the maximal operator in Orlicz spaces (see \cite[Section 4.3]{HH})}.

\begin{definition}\label{maximal}
Given an open set $\Omega\subseteq\R^d$ and $f\in L^1_{\rm loc}(\Omega)$, the \emph{(centered) Hardy-Littlewood maximal operator} is $Mf:\Omega\to[0,\infty]$ defined as
\begin{equation}
    Mf(x) := \sup_{\rho>0}\frac{1}{\mathcal{L}^d(B_\rho(x))}\int_{B_\rho(x)\cap \Omega}|f(y)|\,\d y.
\label{eq:maximop}
\end{equation}
\noindent Analogously, for $\nu$ a positive, finite Radon measure in $\R^d$ one can define
\begin{equation*}
M\nu(x):= \sup_{\rho>0} \frac{\nu(B_\rho(x))}{\mathcal{L}^d(B_\rho)}\,,\quad x\in \R^d\,.
\end{equation*}
\end{definition}

\noindent As a consequence of the Besicovitch covering theorem (see, e.g., \cite{EG}), it can be shown that
\begin{equation*}
\mathcal{L}^d(\{x\in\R^d:\,\, M\nu(x)>\lambda\}) \leq \frac{c}{\lambda} \nu(\R^d)
\label{eq:2.3dclv}
\end{equation*}
for a constant $c$ depending only on $d$.

\noindent
It is also well-known that the maximal operator is bounded on $L^p$, for $p>1$. Next theorem shows that maximal operator is bounded {on $L^\varphi$} provided the weak $\Phi$-function {$\varphi$} satisfies \inc{p} { for some $p>1$}.

\begin{proposition} \label{prop:boundedmaxop}
Let an open set $\Omega\subseteq\R^d$ and $\varphi\in\Phiw$ be given. If $p>1$ and $\varphi$ satisfies \inc{p}, then there exists a $\mu>0$ such that
\[
    \varphi(\mu\, Mf(x))^\frac{1}{p}{\leq} M\left(\varphi(|f|)^\frac{1}{p}\right)(x)
\]
for every ball $B$, $x\in B\cap \Omega$, and $f\in L^{\varphi}(\Omega)$ satisfying $\displaystyle{\int_\Omega\varphi(|f|)\,\d x\le1}$.
\end{proposition}
\begin{remark}
{Proposition \ref{prop:boundedmaxop} is a particular case of ~\cite[Corollary 4.3.3]{HH}, since $\varphi$ does not depend on the $x$ variable. Inspecting its proof one can say that the constant $\mu$ depends on $\varphi$ in terms of $\varphi^{-1}(1)$. }
\end{remark}

\begin{corollary}\label{C:MaxOpBnd}
Let $\varphi\in\Phiw$ be satisfying \inc{p} and \dec{q}, with $1<p\le q<+\infty$. Then there exists ${C}={C}(\varphi^{-1}(1),d,p,q)$ such that
\[
\int_\Omega \varphi(Mf)\,\d x\le {C} \int_\Omega \varphi(|f|)\,\d x
\]
for every  $f\in L^{\varphi}(\Omega)$ satisfying $\displaystyle{\int_\Omega\varphi(|f|)\,\d x\le1}$.
\end{corollary}
\begin{proof}
Noticing that 
$\varphi(t)\le \varphi(\mu t)\max\{\frac{1}{\mu^p},\frac{1}{\mu^q}\}$, the previous proposition produces the result.
\end{proof}

\subsection{$BV$ and $SBV$ functions}\label{s:bv}

For a general survey on the spaces of $BV$ and $SBV$ functions 
we refer for instance to \cite{Ambrosio-Fusco-Pallara:2000}. Below, we just recall some basic definitions useful in the sequel. 

If  $u\in L^1_{\rm loc}(\Omega)$ and $x\in\Omega$, the {\it precise representative of $u$ at $x$} is defined
as the unique value $\widetilde{u}(x)\in\R$ such that
$$
\lim_{\rho\to 0^+} \frac{1}{\rho^d}\int_{{B_\rho(x)}}\!|u(y)-\widetilde{u}(x)|\,\d x=0\,.
$$
The set of points in $\Omega$ where the precise representative of $x$ is not defined is called the
{\it approximate singular set} of $u$ and denoted by $S_u$. We say that a point $x\in\Omega$ is an approximate jump point of $u$
if there exist $a,b\in\R$ and $\nu\in\mathbb{S}^{d-1}$, such that $a\not = b$ and
$$
\lim_{\rho\to 0^+}\dashint_{B^+_\rho(x,\nu)} |u(y)-a|\, \d y=0
\qquad{\rm and}\qquad
\lim_{\rho\to 0^+}\dashint_{B^-_\rho(x,\nu)} |u(y)-b|\, \d y=0
$$
where $B^\pm_\rho(x,\nu):= \{y\in B_\rho(x)\ :\ \langle y-x,\nu\rangle\gtrless0\}$.
The triplet $(a,b,\nu)$ is uniquely determined by the previous formulas, up to a permutation
of $a,b$ and a change of sign of $\nu$, and it is denoted by $(u^+(x),u^-(x),\nu_u(x))$.
The Borel functions $u^+$ and $u^-$  are called the {\it upper and
lower approximate limit} of $u$ at the point $x\in\Omega$. The set of approximate jump points of $u$ is denoted by $J_u\subseteq S_u$.

The space ${BV}(\Omega)$ of {\it functions of bounded variation} is defined as the set
of all $u\in L^1(\Omega)$ whose distributional gradient $Du$ is a bounded Radon measure on $\Omega$
with values in $\R^d$.  Moreover, the usual decomposition
\begin{equation*}
Du = \nabla u\,\mathcal{L}^d + D^c u + (u^+-u^-)\otimes \nu_u\,\mathcal{H}^{d-1}\ristretto{J_u}
\end{equation*}
\noindent
holds, where $\nabla u$ is the Radon-Nikod\'ym derivative of $Du$ with respect to the Lebesgue measure and $D^cu$ is  the {\it Cantor part} of $Du$. {If $u\in BV(\Omega)$, then $\nabla u(x)$ is the  \emph{approximate gradient} of $u$ for a.e. $x\in\Omega$:
\[
\lim_{\rho\to 0}\-int_{B_\rho(x)}\frac{|u(y)-u(x)-\nabla u(x)(y-x)|}{|y-x|}\,\mathrm{d}y=0\,.
\]}
For the sake of simplicity, we denote by $D^su =D^c u + (u^+-u^-)\otimes \nu_u\,\mathcal{H}^{d-1}\ristretto{J_u}$.
{{If $u\in{BV}(\Omega)$, then $\mathcal{H}^{d-1}(S_u\setminus J_u)=0$; so in the sequel we shall essentially identify the two sets.}}

We recall that the space ${SBV}(\Omega)$ of {\it special functions of bounded variation} is defined as the set
of all $u\in BV(\Omega)$ such that $D^su$ is concentrated on $S_u$; i.e., $|D^su|(\Omega\setminus S_u)=0$. Finally, for $p>1$ the space $SBV^p(\Omega)$ is the set of $u\in SBV(\Omega)$ with $\nabla u\in L^p(\Omega;\R^d)$ and $\mathcal{H}^{d-1}(S_u)<\infty$.

{In order to recall a Poincar\'e-Wirtinger inequality for $SBV$ functions with small jump set in a ball, we first fix some notation useful also in the sequel. With given $a$, $b\in\R$, we denote $a\wedge b:=\min(a,b)$ and $a\vee b:=\max(a,b)$. Let $B$ be a ball in $\R^d$. For every measurable function $u:B\to\R$, we set
\begin{equation*}
u_*(s;B):= \inf\{t\in\R:\,\, \mathcal{L}^d(\{u<t\}\cap B)\geq s\} \qquad \mbox{ for } 0\leq s \leq \mathcal{L}^d(B),
\end{equation*}
and
\begin{equation*}
\quad {\rm med}(u;B):=u_*\left(\frac{1}{2}\mathcal{L}^d(B);B\right).
\end{equation*}}

{For every $u\in SBV(\Omega)$ such that
\begin{equation*}
\left(2\gamma_{\rm iso}\mathcal{H}^{d-1}(S_u\cap B)\right)^{\frac{d}{d-1}} \leq \frac{1}{2}\mathcal{L}^d(B)\,,
\end{equation*}
we define
\begin{equation*}
\begin{split}
\tau'(u;B) & := u_*\left(\left(2\gamma_{\rm iso}\mathcal{H}^{d-1}(S_u\cap B)\right)^{\frac{d}{d-1}};B\right)\,, \\
\tau''(u;B) & := u_*\left(\mathcal{L}^d(B)-\left(2\gamma_{\rm iso}\mathcal{H}^{d-1}(S_u\cap B)\right)^{\frac{d}{d-1}};B\right)\,, 
\end{split}
\end{equation*}
and the truncation operator
\begin{equation}
T_Bu(x):= (u(x)\wedge \tau''(u;B)) \vee \tau'(u;B)\,,
\label{eq:truncated}
\end{equation}
where $\gamma_{\rm iso}$ is the dimensional constant in the relative isoperimetric inequality. }



{We are now in position to state the aforementioned Poincar\'e-Wirtinger inequality, due to De Giorgi-Carriero-Leaci (see \cite[Theorem~3.1]{DeGCarLea} for the original proof in the scalar setting, and \cite[Theorem~2.5]{CL} for the subsequent extension to vector-valued functions).}

{\begin{theorem}
Let $u\in SBV(B)$ and assume that
\begin{equation}
\left(2\gamma_{\rm iso}\mathcal{H}^{d-1}(S_u\cap B)\right)^{\frac{d}{d-1}} \leq \frac{1}{2}\mathcal{L}^d(B)\,.
\label{(10)bis}
\end{equation}
If $1\leq p < d$ then the function $T_Bu$ satisfies $|DT_Bu(B)|\le 2\int_B|\nabla u|\,\mathrm{d}y$, 
\begin{equation*}
\left(\int_B |T_Bu-{\rm med}(u;B)|^{p^*}\,\mathrm{d}x\right)^\frac{1}{p^*} \leq \frac{2\gamma_{\rm iso}p(d-1)}{d-p} \left(\int_B|\nabla u|^p\,\mathrm{d}x\right)^\frac{1}{p},
\label{(11)bis}
\end{equation*}
and
\begin{equation}
\mathcal{L}^d(\{T_Bu\neq u\}\cap B) \leq 2 \left(2\gamma_{\rm iso}\mathcal{H}^{d-1}(S_u\cap B)\right)^{\frac{d}{d-1}}\,,
\label{(12)bis}
\end{equation}
where $p^*:=\frac{dp}{d-p}$.
If $p\ge d$, then, for any $q\ge 1$,
\begin{equation*}\label{maggdbis}
{\left(\int_B |T_Bu-{\rm med}(u;B)|^{q}\,\mathrm{d}x\right)^\frac{1}{q} \leq c(q,N,\gamma_{\rm iso})(\mathcal{L}^d(B))^{\frac{1}{q}+\frac{1}{d}-\frac{1}{p}} \left(\int_B|\nabla u|^p\,\mathrm{d}x\right)^\frac{1}{p}\,.}
\end{equation*}
\label{thm:poincsbv}
\end{theorem}}

{As a first application of Theorem~\ref{thm:poincsbv} one can obtain the following sufficient condition for the existence of the approximate limit at a given point (see \cite[Theorem 7.8]{Ambrosio-Fusco-Pallara:2000}).
\begin{theorem}
Let $u\in SBV_{\rm loc}(\Omega)$ and $x\in\Omega$. If there exist $p,q>1$ such that
\begin{equation*}
\lim_{\rho\to0} \frac{1}{\rho^{d-1}} \left[\int_{B_\rho(x)}|\nabla u|^p\,\mathrm{d}y + \mathcal{H}^{d-1}(S_u\cap B_\rho(x)) \right] =0 \quad \mbox{and} \quad \mathop{\lim\sup}_{\rho\to0} \dashint_{B_\rho(x)} |u(y)|^q \,\mathrm{d}y <\infty\,,
\end{equation*}
then $x\not\in S_u$.
\label{thm:thm7.8AFP}
\end{theorem}}

\subsection{The space $SBV^{\varphi}$. Sobolev-Poincar\'e inequality and Lusin-type approximation.} \label{sec:poincare}

We denote by $SBV^{\varphi}(\Omega)$ the set of functions $u\in SBV(\Omega)$ with $\nabla u\in L^{\varphi}(\Omega;\mathbb{R}^{d})$ and $\mathcal{H}^{d-1}(S_u)<+\infty$.

{A fundamental ingredient in the proof of our main result will be the following Sobolev-Poincar\'e inequality for $SBV^\varphi$-functions with small jump set, extending Theorem \ref{thm:poincsbv} to the Orlicz setting. 
The result below is a slight refinement of that proven in \cite[Theorem 5.7]{ars} for $r=1$. The case $r>1$ can be inferred with minor modifications, so we briefly sketch the proof. }

\begin{theorem}\label{thm:poincsbvphi}
Let $\varphi$ be a weak $\Phi$-function satisfying \inc{p} and \dec{q}, let $B$ be any ball and $u\in SBV^\varphi(B)$, and assume that
\begin{equation}
\left(2\gamma_{\rm iso}\mathcal{H}^{d-1}(S_u\cap B)\right)^{\frac{d}{d-1}} \leq  \frac{1}{2}\mathcal{L}^d(B)\,.
\label{(10)}
\end{equation}
Then the function $T_Bu$ satisfies $|DT_Bu(B)|\le 2\int_B|\nabla u|\,\mathrm{d}y$, 
\begin{equation}
\mathcal{L}^d(\{T_Bu\neq u\}\cap B) \leq 2 \left(2\gamma_{\rm iso}\mathcal{H}^{d-1}(S_u\cap B)\right)^{\frac{d}{d-1}},
\label{(12)}
\end{equation}
and {there exists $r\in(1,\frac{d}{d-1})$ such that}
\begin{equation}
\left(\int_B \varphi(|T_Bu-{\rm med}(u;B)|)^r\,\mathrm{d}x \right)^{\frac{1}{r}}\leq C \int_B\varphi(|\nabla u|)\,\mathrm{d}x,
\label{(11)}
\end{equation}
where 
$C=C(d,p,q)$.

\end{theorem}
 
\proof
{Let $r\in(1,\frac{pd-1}{p(d-1)}]$ be fixed. In particular, this implies that $r<\frac{d}{d-1}$. } It will suffice to check that the functions
\begin{equation*}
\phi(t) := \varphi^\frac{d}{d-1}(t) t^{-\frac{1}{d-1}} \quad \mbox{ and } \quad \psi(t):=\varphi^r(t)
\end{equation*}
comply with the assumptions of \cite[Lemma~5.8]{ars}. Indeed, for every $0<t\leq1$, recalling that $\varphi^{-1}$ is concave, we have
\begin{equation*}
0\leq \phi(\psi^{-1}(t)) = \frac{t^\frac{d}{(d-1) r}}{[\varphi^{-1}(t^\frac{1}{r})]^\frac{1}{d-1}} \lesssim t^\frac{1}{r}\,,
\end{equation*}
whence $\lim_{t\to0^+}\phi(\psi^{-1}(t))=0$. Moreover, $t\to \frac{\phi(t)}{\psi(t)}$ is increasing, since 
\begin{equation*}
\frac{\phi(t)}{\psi(t)} = \left[\frac{\varphi(t)}{t^\frac{1}{d-r(d-1)}} \right]^{\frac{d}{d-1}-r}\,,
\end{equation*}
and $t\to\frac{\varphi(t)}{t^\frac{1}{d-r(d-1)}}$, with $0<\frac{1}{d-r(d-1)}\leq {p}$, is increasing by Remark~\ref{rem:varinc}. Then \eqref{(11)} can be inferred by arguing as in the proof of \cite[Theorem~5.7]{ars}, so we omit the details.
\endproof


{A consequence of Theorem~\ref{thm:poincsbvphi} is the following compactness result.

\begin{theorem}\label{thm:3.5}
Let $B\subset\Omega$ be a ball, and $\{\varphi_j\}_{j\in\mathbb{N}}\subset\Phiw$ be complying with \inc{p}, \dec{q}, and {\rm (A0)} ($\frac{1}{L}\le\varphi_j(1)\le L$).
{Let $r>1$ be the exponent of Theorem \eqref{thm:poincsbvphi}}. If $\{u_j\}_{j\in\mathbb{N}}\subset SBV^{\varphi_j}(B)$ is such that
\begin{equation}
\Lambda:=\sup_{j\in\mathbb{N}}\int_B \varphi_j(|\nabla u_j|)\,\mathrm{d}y < +\infty\,,\quad \lim_{j\to+\infty} \mathcal{H}^{d-1}(S_{u_j}\cap B)=0,
\label{eq:ebounded}
\end{equation}
then there exist a function $u_0\in W^{1,1}(B)$ and a subsequence (not relabeled) of $\{u_j\}$ such that
\begin{equation}\label{eq:claims}
\begin{split}
&{\rm (i)}  \ {\bar u}_j:=T_B{u}_j-{\rm med}(u_j;B) \to u_0 \quad  \mathcal{L}^d-\hbox{a.e. in $B$}\,,\\
&{\rm (ii)} \int_{B} \varphi_j(|\bar{u}_j|)^r\mathrm{d}y\le C\,, \,\,\,\, {\mbox{ where $C=C(d,p,q,\Lambda)>0$\,,}} \\   
&{\rm (iii)}  \,\, \nabla\bar{u}_j\rightharpoonup \nabla u_0\quad \hbox{weakly in $L^1(B)$}\,,\\
&{\rm (iv)}  \,\, {\mathcal L}^d(\{{\bar  u}_j\neq u_j\}\cap B_1)\le 2 \left(2\gamma_{\rm iso}\mathcal{H}^{d-1}(S_{u_j}\cap B_1)\right)^{\frac{d}{d-1}}\,.
\end{split}
\end{equation}
\end{theorem}

\begin{proof}
Noting that 
\begin{equation*}
\frac{1}{L} \min\{t^p,t^q\} \leq \varphi_j(t)\leq L \max\{t^p,t^q\}\,, 
\end{equation*}
the result is a straightforward  application of Theorem~\ref{thm:poincsbvphi}, compactness theorems in $BV$ and $SBV$ (see \cite[Theorem 3.23 and Theorem 4.7]{Ambrosio-Fusco-Pallara:2000}), and \cite[Remark 7.6]{Ambrosio-Fusco-Pallara:2000}. 
\end{proof}

We conclude this section with the following theorem, proved in \cite[Lemma~7.2]{ars},  {which is a Lusin-type approximation result in $SBV^\varphi$, as it }concerns the approximation of $SBV^\varphi$ functions with Lipschitz functions in the unit ball.
\begin{theorem}\label{Lusin}
{Let $\varphi\in\Phiw(\Omega)$ be satisfying \inc{p}.} For every $u\in SBV^\varphi(B_1)\cap L^\infty(B_1)$ and every $\lambda > 0$ there exists a Lipschitz function $u_\lambda : B_1\to\mathbb{R}$ satisfying {\rm Lip}$(u_\lambda) \le c\,\lambda$ with $c=c(d)$, such that $u_\lambda = u$ in
$\{M |Du|\le \lambda\}$ and
\[
{\mathcal L}^d(A\cap \{M |Du| > \lambda\})\le 2\frac{c}{\lambda}\|u\|_{L^\infty(B_1)}\mathcal{H}^{d-1}(S_u) + \frac{1}{\varphi_{B_1}^-(\lambda)}\int_{\{M|\nabla u|>\lambda\}\cap A}\varphi_{B_1}^-(M |\nabla u|)\,\d x,
\]
for any Borel set $A\subset B_1$, {where $M$ is introduced in Definition \ref{maximal}.}
\end{theorem}

\subsection{Auxiliary results}\label{sec:altripreli}

{This section collects several supporting results used in the proof of the decay lemma (Section \ref{sec:decay}).} 

\noindent First, we recall  a regularity result from \cite{DSV}
for Sobolev minimizers of autonomous variational integrals (see Lemma~5.8 therein), which will be used in our proof. As noticed in \cite[Lemma 4.12]{HastoOk}, where a more general result is proved, the constant $C_0$ depends on $\varphi$ only through $p$ and $q$.   

\begin{proposition}\label{prop:nondip}
{Let $\varphi\in \Phic\cap C^1([0,+\infty))\cap C^2((0,+\infty))$ be with $\varphi'$ satisfying \inc{p-1} and \dec{q-1} for some $1<p\le q$.} Let $B_R$ be a ball in $\mathbb{R}^d$ and let $v\in W^{1,\varphi}(B_R)$ be a local minimizer of the functional $w\mapsto\int_{B_R}\varphi(|\nabla w|)\,\mathrm{d}x$. Then there exists a constant $C_0=C_0(p,q,d)$ such that
\begin{equation}\label{Linfty}
\sup_{B_{R/2}}\varphi(|\nabla v|)\le C_0\-int_{B_R}\varphi(|\nabla v|)\mathrm{d}x\,.
\end{equation}
\end{proposition}

\begin{remark}\label{senzac2}
Actually, the $C^2$ regularity assumption on $\varphi$ can be removed working as in  \cite{HastoOk}, finding for every $\varepsilon>0$ a suitably regular (twice differentiable) auxiliary function $\varphi^\varepsilon$ and approximating the minimizer with the solution to the related 
minimization problem, identified by $\varphi^\varepsilon$. \\
{Let $\varphi\in C^1([0,+\infty))$. Let $\eta\in C^\infty_c(\R)$ be such that $\eta\geq0$, ${\rm supp}\,\eta\subset (0,1)$ and $\|\eta\|_1=1$, and consider for $\varepsilon>0$ the scaled kernel $\eta_\varepsilon(t):=\frac{1}{\varepsilon}\eta(\frac{t}{\varepsilon})$. According to \cite{HastoOk} we define
\begin{equation*}
\varphi^\varepsilon(t):= \int_0^{+\infty} \varphi(t\sigma) \eta_\varepsilon(\sigma-1)\,\mathrm{d}\sigma \,.
\end{equation*}
Then it easy to check that $\varphi^\varepsilon\in C^\infty((0,+\infty))\cap C^1([0,+\infty))$, 
\begin{equation}
\varphi(t) \leq \varphi^\varepsilon(t) \leq (1+\varepsilon)^q \varphi(t)
\label{eq:epsestim}
\end{equation}
and $(\varphi^\varepsilon)'$ satisfies {\rm (inc)}$_{p-1}$ and {\rm (dec)}$_{q-1}$.}

{Now, let $u\in W^{1,\varphi}(B_R)$ be a local minimizer of the functional $w\mapsto\int_{B_R}\varphi(|\nabla w|)\,\mathrm{d}x$ and $B_r \subset\subset B_R$. Note that, by \eqref{eq:epsestim}, it holds that $u\in W^{1,\varphi^\varepsilon}(B_r)$. Let $u_\varepsilon\in W^{1,\varphi^\varepsilon}(B_r)$ be the minimizer of
\begin{equation} \int_{B_r}\varphi^\varepsilon(|\nabla v|)\,\mathrm{d}x \,\,\,\, \mbox{ such that $u_\varepsilon=u$ on $\partial B_r$,}
\label{eq:regularizedprob}
\end{equation}
whose existence is ensured by the {direct methods in the calculus of variations.} Then, from the minimality of $u_\varepsilon$ 
and taking into account \eqref{eq:epsestim} we get
\begin{equation}
\int_{B_r}{\varphi}(|\nabla u_\varepsilon|)\,\mathrm{d}x \leq \int_{B_r}\varphi^\varepsilon(|\nabla u_\varepsilon|)\,\mathrm{d}x \leq \int_{B_r}\varphi^\varepsilon(|\nabla u|)\,\mathrm{d}x \leq (1+\varepsilon)^q \int_{B_r}{\varphi}(|\nabla u|)\,\mathrm{d}x \,,
\label{eq:epsestim2}
\end{equation}
whence we deduce that the sequence $u_\varepsilon$ is equibounded in $W^{1,\varphi}(B_r)$. Then, since by virtue of Proposition~\ref{prop:properties}(i) $\varphi$ satisfies \inc{p} and \dec{q}, the space $W^{1,\varphi}$ is reflexive \cite[Theorem 6.1.4(d)]{HH}. Therefore, $u_\varepsilon$ weakly converges (up to a subsequence not relabeled) in $W^{1,\varphi}(B_r)$ to a function $\bar{u}$ such that $\bar{u}=u$ on $\partial B_r$.  Now, by lower semicontinuity we get
\begin{equation}
\int_{B_r}{\varphi}(|\nabla \bar{u}|)\,\mathrm{d}x \leq \mathop{\lim\inf}_{\varepsilon\to0}\int_{B_r}{\varphi}(|\nabla u_\varepsilon|)\,\mathrm{d}x \leq \int_{B_r}{\varphi}(|\nabla v|)\,\mathrm{d}x
\label{eq:epsestim3}
\end{equation}
for every function $v$ coinciding with $u$ on $\partial B_r$. Thus, $\bar{u}$ is a local minimizer of $w\mapsto\int_{B_r}\varphi(|\nabla w|)\,\mathrm{d}x$, and, by uniqueness of solutions to the boundary value problem, it holds that $\bar{u}=u$ on $B_r$. Combining \eqref{eq:epsestim2} and \eqref{eq:epsestim3} we also get
\begin{equation*}
\lim_{\varepsilon\to0} \int_{B_r}{\varphi}(|\nabla u_\varepsilon|)\,\mathrm{d}x = \int_{B_r}{\varphi}(|\nabla u|)\,\mathrm{d}x \,.
\end{equation*} 
Now, by \cite[Lemma 4.12 (4.13)]{HastoOk}, for any $B_\rho(x_0)\subset B_r$ we have
\begin{equation*}
\sup_{B_{\frac{\rho}{2}}(x_0)} |\nabla u_\varepsilon| \leq C \-int_{B_\rho(x_0)}|\nabla u_\varepsilon|\mathrm{d}x \leq C \varphi^{-1} \left(\-int_{B_\rho(x_0)} \varphi(|\nabla u_\varepsilon|)\mathrm{d}x\right) \,,
\end{equation*}
where the constant $C$ depends only on $p,q$ and $d$. Passing to the limit as $\varepsilon\to0$ we finally get 
\begin{equation*}
\sup_{B_{\frac{\rho}{2}}(x_0)} |\nabla u|\leq C \varphi^{-1} \left(\-int_{B_\rho(x_0)} \varphi(|\nabla u|)\mathrm{d}x\right) \,.
\end{equation*}}
\end{remark}

Finally, we consider the $\varphi$-recession function associated to a sequence of convex functions $\varphi_j$, capturing the behaviour at infinity of $\varphi_j$ (see also  \cite[Lemma 3.2]{FMT} and \cite[Lemma 4.3]{LSSV}). \begin{lemma}\label{lem:lemma3.2fmt}
Let $(\varphi_j)_{j\in\N}$, $\varphi_j: [0,+\infty)\to[0,+\infty)$, be a sequence of $C^1$ convex functions satisfying $\varphi_j(0)=0$ and assume that $\varphi'_j$ satisfies \inc{p-1} and \dec{q-1}, where $1<p\le q<+\infty$.
Let $(\beta_j)\subset(0,\infty)$ be a sequence such that $\lim_j \beta_j=+\infty$. Then, setting
\begin{equation*}
\widetilde{\varphi}_j(t):=\frac{\varphi_{j}(t\beta_{j})}{\varphi_j(\beta_j)}\,,\quad t\in[0,+\infty)\,,\,\, j\in\N\,,
\end{equation*}
there exists a subsequence $(\beta_{j_k})$ such that $\widetilde{\varphi}_{j_k}$ converge to a  $C^1$ convex function $\varphi_\infty$ uniformly on compact subsets of $[0,+\infty)$. Moreover, $\varphi'_\infty$ satisfies \inc{p-1} and \dec{q-1}.
\end{lemma}
\begin{proof}
The assertion follows in a standard way noticing that each $\widetilde{\varphi}_j$  is uniformly bounded and equicontinuous in any compact set of $[0,+\infty)$.  Indeed, for every $t\ge 0$, thanks to \eqref{cons2}, 
\[
{\widetilde \varphi}_j(t)\le \max\{t^{p},t^{q}\},
\]
which entails the equiboundedness (with respect to $j$) on compact sets of $[0,+\infty)$.
Moreover, given $t>t_0\ge 0$, from \eqref{cons1} 
and \eqref{cons2}, we have
\[
{\widetilde \varphi}_j(t)-{\widetilde \varphi}_j(t_0)=\int_{t_0}^t\frac{{\mathrm d}}{{\mathrm d}\tau}{\widetilde \varphi}_j(\tau){\mathrm d}\tau=\int_{t_0}^t\frac{\varphi_j'(\tau\beta_j)\beta_j}{\varphi_j(\beta_j)}{\mathrm d}\tau\le \int_{t_0}^t \frac{q\varphi_j(\tau\beta_j)}{\varphi_j(\beta_j)\tau}{\mathrm d}\tau\le q\int_{t_0}^t\max\{\tau^{p-1},\tau^{q-1}\}{\mathrm d}\tau,
\]
from which we deduce the equicontinuity of $\tilde{\varphi}_j$ on compact sets of $[0,+\infty)$, since the function $t \to \max\{t^{p-1},t^{q-1}\}\in L^1_{\rm loc}([0,+\infty))$.
Then, {by Ascoli-Arzelà Theorem,} up to a subsequence (not relabeled), ${\widetilde\varphi}_j$ converges to a function $\varphi_\infty$, uniformly on compact sets of $[0,+\infty)$. Obviously, the continuous function $\varphi_\infty$ inherits the convexity of $\varphi_j$.
Moreover, we can prove that $\varphi_\infty$ is a $C^1$ function. Actually, the regularity properties of $\varphi'_j$ allow us to improve the convergence of $\widetilde{\varphi}_j$, entailing the local uniform convergence of ${\widetilde\varphi}'_j$. From now on let us consider the convergent subsequence of $\widetilde{\varphi}_j$, not relabeled.

\noindent Since, for $t>0$, { again by \eqref{cons1} and \eqref{cons2} we have}
\begin{equation}\label{boundphi'j}
{\widetilde\varphi}'_j(t)=\frac{\beta_j\varphi'_j(t\beta_j)}{\varphi_j(\beta_j)}=\frac{t\beta_j\varphi'_j(t\beta_j)}{\varphi_j(t\beta_j)}\,\frac{\varphi_j(t\beta_j)}{t\varphi_j(\beta_j)}\le q\frac{\max\{t^p,t^q\}}{t}= q\max\{t^{p-1},t^{q-1}\},
\end{equation}
{and since ${\widetilde\varphi}'_j(0)=0$,} we get the uniform boundedness of ${\widetilde\varphi}'_j$ in any compact set of $[0,+\infty)$. 

\noindent Now, for $t>t_0>0$, we use the property \dec{q-1} of $\varphi'_j$ (and accordingly of $\widetilde{\varphi}'_j$) and \eqref{boundphi'j}, to obtain
\[
{\widetilde\varphi}'_j(t)-{\widetilde\varphi}'_j(t_0)\le {\widetilde\varphi}'_j(t_0)\left[\left(\frac{t}{t_0}\right)^{q-1}-1\right]\le q\max\{t_0^{p-1},t_0^{q-1}\}\left[\frac{t^{q-1}-t_0^{q-1}}{t_0^{q-1}}\right],
\]
which, together with the fact that 
\[
{\widetilde\varphi}'_j(t)-{\widetilde\varphi}'_j(0)={\widetilde\varphi}'_j(t) \le q\max\{t^{p-1},t^{q-1}\}, \ \ \ \ t>0,
\]
lead to the equicontinuity of ${\widetilde\varphi}'_j$ on compact sets of $[0,+\infty)$.
Then $\varphi_\infty$ is a $C^1$ function (obtained as the uniform limit on compact sets of a subsequence of ${\widetilde\varphi}'_j$). Finally, we observe that $\varphi'_\infty$ satisfies \inc{p-1} and \dec{q-1}.
\endproof

\section{Free-discontinuity functionals with non-standard growth}\label{sec:freepx}
In this paragraph we consider integral functionals of the form 
\begin{equation}\label{funct}
F(u,c,A) :=\int_A\varphi(x,|\nabla u|)\,\mathrm{d}x+c{\mathcal H}^{d-1}(S_u\cap A),
\end{equation}
defined on $SBV_{\rm loc}(\Omega)$, {where $c>0$ and $A\subset\Omega$ is an open set.} The function $\varphi:\Omega\times[0,+\infty)\to[0,+\infty)$  satisfies the following  assumptions:
{
\begin{enumerate}[font={\normalfont},label={(${\rm H}_\arabic*$)}]
\item $\varphi\in\Phic(\Omega)\cap C^1([0,+\infty))$; \label{hpuno}
\item  $\varphi_t$ satisfies \azero{}, \inc{p-1}, \dec{q-1}, where $1<p\le q$; \label{hpdue}
\item $\varphi$ satisfies \vauno{}.  \label{hptre}
\end{enumerate}}

{Since \ref{hpuno} and \ref{hpdue} are in force, by virtue of Proposition~ \ref{prop:properties} we conclude that $\varphi\in N(\Omega)$ and it satisfies \eqref{cons2}-\eqref{v0phi}.} Concerning the modulus of continuity of $\omega(r)$ in \vauno{} for our purposes it will be enough to assume that $\omega(0)=0$ without prescribing any decay rate. We remark that in the Sobolev case maximal regularity has been obtained in \cite{HastoOk} under the assumption $\omega(r)\lesssim r^\beta$, for some  $\beta\in(0,1]$. This additional assumption is however not needed for Ahlfors type regularity of the jump set, as in the blow-up procedure one only needs to be in a position for using \eqref{Linfty}. Clearly, it may come again into play for having regularity of $u$ in $\Omega\setminus K$, which then directly stems out of the result of \cite{HastoOk}.
%
%

%
%
%
%
%


{We recall that a function $u\in SBV_{\rm loc}(\Omega)$ satisfying $F(u,c,A)<+\infty$ for all open sets $A\ssubset\Omega$ is a \emph{local minimizer} of $F(\cdot,c,\Omega)$ in $\Omega$ if 
\[
F(u,c,A)\le F(v,c,A)
\]
for all $v\in SBV_{\rm loc}(\Omega)$ satisfying $\{v\neq u\} \ssubset A \ssubset \Omega$. Then, in order to get an estimate of how far a function $u$ is from being a minimizer of $F$ in $\Omega$, the classical definition of \emph{deviation from minimality} ${\rm Dev}(u,c,\Omega)$ has been introduced (see, e.g., \cite{Ambrosio-Fusco-Pallara:2000}): it is defined as the smallest $\lambda\in [0,+\infty]$ such that
\[
F(u,c,A)\le F(v,c,A) + \lambda
\]
for all $v\in SBV_{\rm loc}(\Omega)$ satisfying $\{v\neq u\} \ssubset A \ssubset \Omega$. 
Clearly, $u$ is a local minimizer of $F(\cdot,c,\Omega)$ in $\Omega$ if ${\rm Dev}(u,c,\Omega)=0$. }

\end{proof}
\subsection{A decay lemma}  \label{sec:decay}

In this Section we prove  a crucial decay property of the energy $F$ {(see \eqref{funct})} in small balls. The proof is  the non-standard growth counterpart of the well-known argument, based on a blow-up procedure, devised for energies with $p$-growth (see, e.g., \cite[Lemma~7.14]{Ambrosio-Fusco-Pallara:2000}).




Recall that the shorthand $F(u,A)$ below stands for $F(u,1,A)$, as defined in \eqref{funct}.

\begin{lemma}[Decay estimate]\label{lem:decay}
Let $\varphi$ be a function satisfying {\rm \ref{hpuno}}, {\rm\ref{hpdue}}, {\rm\ref{hptre}}. There is a constant $C_1=C_1(d,p,q)$ with the property that, for every $\Omega_0\ssubset\Omega$ and for every $\tau\in (0,1)$ there exist $\varepsilon=\varepsilon(\tau,\Omega_0)$, $\theta=\theta(\tau,\Omega_0)$ in $(0,1)$ such that if $u\in SBV(\Omega)$ satisfies, for $x\in\Omega_0$ and $B_\rho(x)\ssubset\Omega$, $\rho<\varepsilon^2$,
\[
F(u,B_\rho(x))\le\varepsilon\rho^{d-1}, \ \ \ \ \ {\rm Dev}(u,B_\rho(x))\le\theta F(u,B_\rho(x)),
\]
then 
\begin{equation}\label{decade}
F(u,B_{\tau\rho}(x))\le C_1\tau^d F(u, B_\rho(x)).
\end{equation}
\end{lemma}

\begin{proof}
It is enough to assume $\tau\in (0,1/2)$ (otherwise just take $C_1=2^d$). We argue by contradiction and assume that \eqref{decade} does not hold. 
In this case, there exist a sequence $u_j\in SBV(\Omega)$, sequences of nonnegative numbers $\varepsilon_j$, $\theta_j$, $\rho_j$, with $\lim_j \varepsilon_j=\lim_j\theta_j=0$, $\rho_j\le \varepsilon_j^2$, and $x_j\in\Omega_0$, with $B_{\rho_j}(x_j)\ssubset\Omega$, such that
\begin{equation}\label{assu1}
F(u_j,B_{\rho_j}(x_j))\le\varepsilon_j\rho_j^{d-1}\,, \ \ \ \ \ {\rm Dev}(u_j,B_{\rho_j}(x_j))\le\theta_j F(u_j,B_{\rho_j}(x_j))\,,
\end{equation}
and 
\begin{equation}\label{assu2}
F(u_j,B_{\tau\rho_j}(x_j))> C_0\tau^d F(u_j, B_{\rho_j}(x_j))\,,
\end{equation}
where $C_0$ comes from  \eqref{Linfty} in Proposition \ref{prop:nondip} {(see Remark~\ref{senzac2})}.  \\
\noindent
{{\it Step 1: Blow-up.}} For every $j$, we consider the $N$-function $\psi_j:[0,+\infty)\to[0,+\infty)$ and the function $w_j:B_1\to\mathbb{R}$ defined as
\begin{equation*}
\psi_j(t):=\varphi(x_j,t) \quad \mbox{ and } \quad w_j(y):=\frac{u_j(x_j+\rho_jy)}{{\sigma_j}\rho_j}\,,
\end{equation*}
respectively, where
\begin{equation}
{\sigma_j:=\psi_j^{-1}\left(\frac{1}{\gamma_j\rho_j}\right)\,, }
\label{eq:sigmaj}
\end{equation}
and $\gamma_j:=\frac{1}{\varepsilon_j}$. We denote by $m_j=$ med$(w_j, B_1)$ and define $v_j =w_j -m_j$;  if we set 
\[
F_j(v,\gamma_j,B_\rho):=\int_{B_\rho}\varphi_j(y,|\nabla v|)\,\mathrm{d}y+\gamma_j\mathcal{H}^{d-1}(S_v,B_\rho),
\]
with ${\varphi_j(y,t):=\gamma_j\rho_j\varphi(x_j+\rho_jy,t\sigma_j)}$, \eqref{assu1} and \eqref{assu2} can be rewritten respectively as
\begin{equation}\label{assv1}
F_j(v_j,\gamma_j,B_1)\le 1\,, \ \ \ \ \ {\rm Dev}_j(v_j,\gamma_j,B_1)\le\theta_j\,,
\end{equation}
and 
\begin{equation}\label{assv2}
F_j(v_j,\gamma_j,B_\tau)> C_0\tau^d F_j(v_j,\gamma_j, B_1)\,.
\end{equation}
The first bound in \eqref{assv1} in turn implies 
\begin{equation}\label{assv3}
{\int_{B_1}\varphi_j^-(|\nabla v_j|)\,\mathrm{d}y\le 1, \quad \mbox{ and }\quad \mathcal{H}^{d-1}(S_{v_j}\cap B_1)\le\frac{1}{\gamma_j}=\varepsilon_j\,, }
\end{equation}
where $\displaystyle{\varphi_j^-(t):=\inf_{B_{1}}\varphi_{j}(y,t)}$. The function $\varphi_j^-$ is a weak $\Phi$-function satisfying \inc{p}, \dec{q}, and, as a consequence, \eqref{cons2}; that is,
\[
\varphi_j^-(s)\min\{t^p,t^q\}\le \varphi_j^-(st)\le\max\{t^p,t^q\}\varphi_j^-(s)\,,\quad {\forall s,t.}
\]
In particular, for $j$ sufficiently large, 
\begin{equation}
\min\{t^p,t^q\}\lesssim \varphi_j^-(t)\le \max\{t^p,t^q\},
\label{eq:stimevarphij}
\end{equation}
the constant hidden in $\lesssim$ being independent of $j$.
In fact $\varphi_j^-(1)\le 1$, and, defining $\displaystyle\varphi_j^+(t) :=\sup_{B_1}\varphi_j(y,t)$, { since from the definition
\begin{equation*}
\varphi^\pm_j(t) = \gamma_j\rho_j \varphi^\pm_{\rho_j}(t\sigma_j)\,, 
\end{equation*}}
by assumption \vauno{},
\begin{equation}\label{a01}
\varphi_{j}^-(1)\ge\frac{1}{2} \varphi_{j}^+(1)\ge \frac{1}{2}
\end{equation}
if $\varphi_{\rho_j}^-(\sigma_j)\in[\omega(\rho_j), \frac{1}{\mathcal{L}^d(B_{\rho_j})}]$. For $j$ large enough,  $\varphi_{\rho_j}^-(\sigma_j)<\omega(\rho_j)\le 1$ does not occur since,  
by \eqref{cons2} and \eqref{v0phi}, this would entail $\frac{1}{\gamma_j\rho_j}$ equibounded. If in the end  $\varphi_{\rho_j}^-(\sigma_j)>\frac{1}{\mathcal{L}^d(B_{\rho_j})}$, then
\begin{equation}\label{a02}
\varphi_{j}^-(1) {>}\gamma_j\rho_j\frac{1}{\kappa_d\rho_j^d}=\frac{1}{\kappa_d\varepsilon_j\rho_j^{d-1}}\ge 1,
\end{equation}
for $j$ large enough. Analogously, $\varphi_j^+$ satisfies \inc{p}, \dec{q}, and, {for $j$ sufficiently large,}
\begin{equation*}
\min\{t^p,t^q\}\le \varphi_j^+(t)\lesssim \max\{t^p,t^q\},
\end{equation*}
the constant hidden in $\lesssim$ being independent of $j$.
Since $0$ is a median for all $v_j$, taking into account Theorem \ref{thm:3.5} (applied with $\varphi_j=\varphi_j^-$), and extracting eventually a subsequence (not relabeled for convenience),  we find a function $v_0\in W^{1,1}(B_1)$ such that the following scheme  holds
\begin{equation}\label{converge}
\begin{split}
 &{\rm (i)} \,\, \bar{v}_j:=T_{B_1}{v}_j \to v_0 \quad  \mathcal{L}^d-\hbox{a.e. in $B_1$}\,,\\
 &{\rm (ii)} \,\,{\exists\, r\in(1,{\textstyle\frac{d}{d-1}}) \,\, \mbox{ such that \,\,}\int_{B_1} \varphi_j^-(|\bar{v}_j|)^r\mathrm{d}y\le C,}\\
&{\rm (iii)}  \,\, \nabla\bar{v}_j\rightharpoonup \nabla v_0\quad \hbox{weakly in $L^1(B_1)$}\,,\\
&{\rm (iv)} \,\, {\mathcal L}^d(\{{\bar  v}_j\neq v_j\}\cap B_1)\le 2 \left(2\gamma_{\rm iso}\mathcal{H}^{d-1}(S_{v_j}\cap B_1)\right)^{\frac{d}{d-1}}\underset {j\to+\infty}{\longrightarrow} 0\,.
\end{split}
\end{equation}
Let us observe that \eqref{converge}(iv) and \eqref{assv3} imply that
\begin{equation}\label{limite}
\lim_{j\to+\infty}\gamma_j\,{\mathcal L}^d(\{{\bar  v}_j\neq v_j\}\cap B_1)=0,
\end{equation}
since
\begin{equation*}
{\gamma_j\,{\mathcal L}^d(\{{\bar  v}_j\neq v_j\}\cap B_1) \leq 2\left(2\gamma_{\rm iso}{}\right)^{\frac{d}{d-1}}\varepsilon_j^{\frac{1}{d-1}}\,. }
\end{equation*}
{{\it Step 2: Modification of ${\bar v}_j$.}}
Introducing a further truncation, depending on $\varphi$ and on the ball $B_{\rho_j}(x_j)$, we modify the sequence ${\bar v}_j$ in order to extend the Sobolev-Poincar\'e inequality given by {Theorem~\ref{thm:poincsbvphi}} for
$SBV^{\varphi}$ functions with small jump set to the cases where $\varphi$ is a generalized Orlicz
function satisfying our assumptions. 

{In order to do that, we introduce the sequence
\begin{equation*}
T_j:= \frac{(\varphi_{\rho_j}^-)^{-1}(1/\rho_j^\mu)}{\sigma_j}
\end{equation*}
for some fixed $\mu\in (1,d)$. Recalling that $\psi_j^{-1}(ts)\ge \min\{s^{\frac{1}{p}}, s^{\frac{1}{q}}\}\psi_j(t)$ for every $s,t$,  we have 
$$
T_j\ge \frac{{\psi_{j}^{-1}(1/\rho_j^{\mu}})}{\sigma_j}\ge \left(\frac{\gamma_j}{\rho_j^{\mu-1}}\right)^{\frac{1}{q}}\,,
$$
so that $\displaystyle\lim_{j\to+\infty}T_j=+\infty$. Moreover, from $\varphi_{\rho_j}^-(T_j\sigma_j)=\frac{1}{\rho_j^\mu}$ we deduce that
\begin{subequations}
\begin{equation}\label{mj2} 
\lim_{j\to+\infty}\varphi_{\rho_j}^-(T_j\sigma_j)\,\rho_j^d=0,
\end{equation}
\begin{equation}\label{mj3} 
\lim_{j\to+\infty}\varphi_{\rho_j}^-(T_j\sigma_j)\,\rho_j=+\infty\,.
\end{equation}
\end{subequations}}
Correspondingly, {for every $j$} we consider the truncation 
\begin{equation*}
\hat{v}_j:= T_j \wedge { {\bar v}_j} \vee (- T_j )\,.
\label{eq:classictruncation}
\end{equation*}
By definition 
\[
|\hat{v}_j|\le T_j\ \ \Longrightarrow\ \  \varphi_{\rho_j}^-(|\hat{v}_j|\sigma_j)\le \varphi_{\rho_j}^-(T_j \sigma_j),
\]
thus, using \eqref{mj2}, the set $\left\{\varphi_{\rho_j}^-(|\hat{v}_j|\sigma_j)>\frac{1}{\mathcal{L}^d(B_{\rho_j})}\right\}\cap B_1$ is empty if $j$ is large. On the other hand, on the set 
$S_j=\left\{\varphi_{\rho_j}^-(|\hat{v}_j|\sigma_j)<\omega(\rho_j)\right\}\cap B_1$, one has
$$
\min\{(|\hat{v}_j|\sigma_j)^p, (|\hat{v}_j|\sigma_j)^q\}\le \omega(\rho_j)Lq\le 1,
$$ 
for $j$ large enough.
Taking these facts into account, and using assumption \vauno{} {on the set $B_1\backslash S_j $},  for {$r>1$ as in  \eqref{converge}(ii) we have } 
\[
\begin{split}
\int_{B_1}\varphi_j^+(|\hat{v}_j|)^r\,\d y&= \int_{S_j}\varphi_j^+(|\hat{v}_j|)^r\,\d y+\int_{B_1\setminus S_j}\varphi_j^+(|\hat{v}_j|)^r\,\d y\\
&\le  (\omega(\rho_j)\,Lq)^{\frac{rp}{q}}\mathcal{L}^d(B_1)+2^r\int_{B_1}\varphi_j^-(|\hat{v}_j|)^r\,\d y\\
&\le (\omega(\rho_j)\,Lq)^{\frac{rp}{q}}\mathcal{L}^d(B_1)+2^r\int_{B_1}\varphi_j^-(|{\bar v}_j|)^r\,\d y\le C
\end{split}
\]
{for $j$ large enough. }
Moreover, {with Chebychev inequality,} 
\[
\gamma_j\,\mathcal{L}^d(\{\hat{v}_j\neq {\bar v}_j\}\cap B_1)\le \gamma_j\,\mathcal{L}^d(\{|{\bar v}_j|\ge T_j\}\cap B_1)\le\frac{C}{\rho_j\varphi_{\rho_j}^-(T_j \sigma_j)}\underset {j\to+\infty}{\longrightarrow} 0,
\]
thanks to \eqref{mj3}, concluding that the sequence $(\hat{v}_j)_{j\in\N}$ satisfies
\begin{equation}\label{convergetronc}
\begin{split}
 &{\rm (i)} \,\, \hat{v}_j \to v_0 \quad  \mathcal{L}^d-\hbox{a.e. in $B_1$}\,,\\
 &{\rm (ii)} \, {\mbox{there exists $j_0$: } } \,\, \int_{B_1} \varphi_j^+(|\hat{v}_j|)^r\mathrm{d}y\le C,\,\, {\mbox{for $j\geq j_0$}} 
\\
&{\rm (iii)}  \,\, \nabla\hat{v}_j\rightharpoonup \nabla v_0\quad \hbox{weakly in $L^1(B_1)$}\,,\\
&{\rm (iv)} \,\, \lim_{j\to+\infty}\gamma_j\,{\mathcal L}^d(\{\hat{v}_j\neq {\bar v}_j\}\cap B_1)=0\,.
\end{split}
\end{equation}
\noindent
{\it Step 3:  The limit map.} We now prove that the function $\varphi_j(y, t)$ converges to an  $N$-function $\varphi_\infty(t)$ 
uniformly on $B_1\times K$, where $K$ is any compact set in {$[0,+\infty)$}. Moreover, $\varphi_\infty$ is a {$C^1([0,+\infty))$} function such that $\varphi'_\infty$ satisfies \inc{p-1} and \dec{q-1}.
Setting
\begin{equation*}
{{\widetilde \psi}_j(t):= \gamma_j\rho_j \psi_j(t\sigma_j) =\gamma_j\rho_j\varphi\left(x_j,t \sigma_j\right)\,, }
\label{eq:scaledftilde}
\end{equation*}
we start by proving that, { with fixed $R>0$, there exists $j_1\geq1$ such that} 
\begin{equation}\label{step1}
|\varphi_j(y,t)-{\widetilde \psi}_j(t)|\le \omega_{j,R}, \quad  {\mbox{for every $(y,t)\in B_1\times [0,R]$, for $j\geq j_1$, }}
\end{equation}
for some 
$\omega_{j,R}$ which is infinitesimal as $j\to+\infty$. In fact, we observe that
\[
\begin{split}
|\varphi_j(y,t)-{\widetilde \psi}_j(t)|&= \gamma_j\rho_j\left|\varphi(x_j+\rho_jy,t\sigma_j)-\varphi(x_j,t\sigma_j)\right|\\
&\le \gamma_j\rho_j\,\omega(\rho_j)\varphi^-_{\rho_j}(t\sigma_j),\\
\end{split}
\]
if $\varphi_{\rho_j}^-(t\sigma_j)\in [\omega(\rho_j),{\textstyle\frac{1}{\mathcal{L}^d(B_{\rho_j})}}]$, thanks to  \ref{hptre}. Recalling the definitions of $\varphi_{\rho_j}^-$ and $\psi_j$, together with \eqref{cons2}, the last term can be estimated as
\[
\gamma_j\rho_j\,\omega(\rho_j)\varphi^-_{\rho_j}(t\sigma_j)\le \gamma_j\rho_j\,\omega(\rho_j)\psi_j(t\sigma_j)\le{\max\{R^p,R^q\}\,\omega(\rho_j)} \,. 
\]
On the other hand, if $\varphi^-_{\rho_j}(t\sigma_j)<\omega(\rho_j)$, from \eqref{cons2} and \eqref{v0phi}, we deduce 
\[
{\min\{(t\sigma_j)^p, (t\sigma_j)^q\}\le Lq \omega(\rho_j)\leq 1, \quad \mbox{ for $j$ large enough, }}
\]
entailing
\[
{t\sigma_j\le (Lq \omega(\rho_j))^{\frac{1}{q}}.}
\]
Then
\[
{|\varphi_j(y,t)-{\widetilde \psi}_j(t)|\le \gamma_j\rho_j 2(Lq \omega(\rho_j))^{\frac{p}{q}}\frac{L}{p}\lesssim \varepsilon_j.  }
\]
Finally, case $\varphi^-_{\rho_j}(t\sigma_j)>\frac{1}{\mathcal{L}^d(B_{\rho_j})}$ cannot occur for $j$ large enough since, {taking into account \eqref{cons2} and \eqref{eq:sigmaj},} it  would lead to  
\[
\frac{1}{\mathcal{L}^d(B_{\rho_j})}<\psi_j(t\sigma_j)\le\max\{t^p,t^q\}\frac{1}{\gamma_j\rho_j}\le\max\{R^p,R^q\}\frac{\varepsilon_j}{\rho_j},
\]
that is 
\[
\frac{1}{\rho_j^{d-1}}< \max\{R^p,R^q\}\kappa_d\,\varepsilon_j,
\]
{which clearly would give a contradiction for $j$ large.}

{Therefore, \eqref{step1} is proven { with $\omega_{j,R}:= \max\left\{\max\{R^p,R^q\}\,\omega(\rho_j), 2(Lq\omega(\rho_j))^{\frac{p}{q}}\frac{L}{p}\varepsilon_j \right\}$ for $j$ large enough}. }

Now, {recalling \eqref{eq:sigmaj}}, thanks to Lemma~\ref{lem:lemma3.2fmt}, applied with {$\beta_j:=\sigma_j$} to the sequence 
{
\begin{equation*}
{\widetilde\psi}_j(t) =\frac{\psi_j(t\sigma_j)}{\psi_j(\sigma_j)} \,,
\end{equation*}
}
we may conclude that up to a subsequence 
{\begin{equation}
\varphi_j(y,t)\to \varphi_\infty(t) \,\, \mbox{ uniformly on $B_1\times K$, where {$K\subset [0,+\infty)$} is compact, }
\label{eq:unifconvphij}
\end{equation}}
for a $C^1([0,+\infty))$ convex function $\varphi_\infty$ such that $\varphi'_\infty$ satisfies \inc{p-1} and \dec{q-1}. \\
\noindent
{\it Step 4: Lower semicontinuity.}  We next turn our attention to the proof of the following lower semicontinuity result:
\begin{equation}\label{lsc}
\displaystyle{\int_{B_1}\varphi_\infty(|\nabla v_0|)\,\d y\le \liminf_{j\to+\infty}\int_{B_1}\varphi_j(y,|\nabla {\hat v}_j|)\,\d y}.
\end{equation}
{We may follow the argument of \cite[Proposition 7.3]{ars}, and even if only a few changes are significant, we prefer to give the details of the proof for the sake of completeness. }

{ First, we notice that it will suffice to show the above inequality for $\nabla {\hat v}_j^T \rightharpoonup \nabla v_0^T$ weakly in $L^1$, where ${\hat v}_j^T$ and $v_0^T$ are the classical truncations at level $T>0$ of ${\hat v}_j$ and $v_0$, respectively, since $\varphi_j(y,0)=0$  on $B_1$. So, up to possibly replacing each ${\hat v}_j$ with the corresponding truncation, we can suppose that $({\hat v}_j)_j$ is bounded uniformly with respect to $j$; namely,
\begin{equation*}
\|{\hat v}_j\|_\infty \leq T\,, \quad \mbox{ for every $j$.}
\end{equation*} }
Thanks to {the first bound in \eqref{assv3}}, we apply Corollary \ref{C:MaxOpBnd} to $\varphi_j^-$, which is a weak $\Phi$-function satisfying \inc{p} and \dec{q}, obtaining
\begin{equation}\label{maxvj}
\int_{B_1}\varphi_j^-(M|\nabla {\hat v}_j|)\,\d y\le \int_{B_1}\varphi_j^-(M|\nabla v_j|)\,\d y\le C,
\end{equation}
having taken into account that, thanks to \eqref{eq:stimevarphij}, $(\varphi_j^-)^{-1}(1)\simeq 1$, and the hidden constants do not depend on $j$. 
{By Chacon's Biting Lemma (see, e.g., \cite[Lemma~5.32]{Ambrosio-Fusco-Pallara:2000}) there exist a sequence of Borel subsets $A_h$ of $B_1$ such that $\mathcal{L}^d(A_h)\to0$ as $h\to+\infty$, and a (not relabelled) subsequence such that $(\varphi_j^-(M|\nabla {\hat v}_j|)\chi_{B_1\setminus A_h})_j$ is equintegrable for every $h\geq1$. }

Let $\lambda>1$. Then, applying Theorem~\ref{Lusin} to ${\hat v}_j$, we find $v_j^\lambda:B_1\to\mathbb{R}$ such that
\begin{equation}
{\rm Lip}(v_j^\lambda)\le c\,\lambda\quad \hbox{ and }\quad v_j^\lambda={\hat v}_j \hbox{ in }B_1\setminus E_j^\lambda,
\label{eq:liptrunc}
\end{equation}
where $E_j^\lambda := \{M|D{\hat v}_j|>\lambda\}$ and 
\begin{equation}\label{ejlambda}
{\mathcal L}^d(E_j^\lambda\setminus A)\le 2\frac{c\,T}{\lambda}\mathcal{H}^{d-1}(S_{\hat{v}_j}) + \frac{1}{\varphi_j^-(\lambda)}\int_{\{M|\nabla {\hat v}_j|>\lambda\}\setminus A}\varphi_j^-(M |\nabla {\hat v}_j|)\,\d y,
\end{equation}
for any Borel set $A\subset B_1$.
Moreover, from Chebychev inequality, \eqref{maxvj}, {\inc{p} for $\varphi_j^-$, and the fact that by \eqref{eq:stimevarphij}, $\varphi_{j}^-(1)\gtrsim 1$ for $j$ large enough,} we deduce
\begin{equation}\label{misura}
{\mathcal{L}^d(\{M|\nabla {\hat v}_j|>\lambda\})\le 
\frac{1}{\varphi_j^-(1)\lambda^p}\int_{\{M|\nabla {\hat v}_j|>\lambda\}}\varphi_j^-(M|\nabla {\hat v}_j|)\,\d y\le \frac{C}{\lambda^p}\,, }
\end{equation}
for $j$ large enough. 
We compute
\[
\begin{split}
\int_{B_1}&\varphi_j(y,|\nabla {\hat v}_j|)\,\d y\ge\int_{B_1\setminus(A_h\cup E_j^\lambda)}\varphi_j(y,|\nabla v_j^\lambda|)\,\d y=\int_{B_1\setminus A_h}\varphi_j(y,|\nabla v_j^\lambda|)\,\d y\\
&-\int_{E_j^\lambda\setminus A_h}\varphi_j(y,|\nabla v_j^\lambda|)\,\d y=\int_{B_1\setminus A_h}\left[\varphi_j(y,|\nabla v_j^\lambda|)-\varphi_\infty(|\nabla v_j^\lambda|)\right]\d y\\
&+\int_{B_1\setminus A_h}\varphi_\infty(|\nabla v_j^\lambda|)\,\d y-\int_{E_j^\lambda\setminus A_h}\varphi_j(y,|\nabla v_j^\lambda|)\,\d y.
\end{split}
\]
Since the convergence \eqref{eq:unifconvphij} implies
\[
\lim_{j\to+\infty}\int_{B_1\setminus A_h}\left[\varphi_j(y,|\nabla v_j^\lambda|)-\varphi_\infty(|\nabla v_j^\lambda|)\right]\,\d y = 0\,,
\]
passing to the liminf in the previous inequality we obtain
\begin{equation}\label{liminf}
\liminf_{j\to+\infty}\int_{B_1}\varphi_j(y,|\nabla {\hat v}_j|)\,\d y\ge \liminf_{j\to+\infty}\int_{B_1\setminus A_h}\varphi_\infty(|\nabla v_j^\lambda|)\,\d y-
\limsup_{j\to+\infty}\int_{E_j^\lambda\setminus A_h}\varphi_j(y,|\nabla v_j^\lambda|)\,\d y.
\end{equation}
We are first dealing with the second term. We have
\[
\int_{E_j^\lambda\setminus A_h}\varphi_j(y,|\nabla v_j^\lambda|)\,\d y\le \int_{E_j^\lambda\setminus A_h}\varphi^+_j(|\nabla v_j^\lambda|)\,\d y.
\]
{In $E_j^\lambda\setminus A_h$ we distinguish between the points of $B_1$ where $\varphi_{\rho_j}^-(|\nabla{v}_j^\lambda|\sigma_j)\in[\omega(\rho_j),1/\mathcal{L}^d(B_{\rho_j})]$, denoting the corresponding set by $S_{j,\lambda}^1$, and the points where that condition does not hold. We then define
\begin{equation*}
S_{j,\lambda}^2:= \left\{\varphi_{\rho_j}^-(|\nabla{v}_j^\lambda|\sigma_j)<\omega(\rho_j)\right\}\cap B_1 \quad \mbox{ and }\quad S_{j,\lambda}^3:=\left\{\varphi_{\rho_j}^-(|\nabla{v}_j^\lambda|\sigma_j)>1/\mathcal{L}^d(B_{\rho_j})\right\}\cap B_1.
\end{equation*}
The set $S_{j,\lambda}^3$ has to be empty for $j$ sufficiently large, as otherwise, using \eqref{cons2} for any fixed point therein, the resulting inequality $\frac{\lambda^q}{\gamma_j\rho_j}>\frac{1}{\kappa_d\rho_j^d}$ would imply $\frac{\gamma_j}{\rho_j^{d-1}}$ uniformly bounded with respect to $j$. }

\noindent In {$S_{j,\lambda}^2$}, thanks to \eqref{cons2} and \eqref{v0phi}, $\min\{(|\nabla{v}_j^\lambda|\sigma_j)^p,(|\nabla{v}_j^\lambda|\sigma_j)^q\}\le Lq\omega(\rho_j)\leq 1$ {for $j$ large enough}, then 
\[
\begin{split}
&\int_{(E_j^\lambda\setminus A_h)\cap S_{j,\lambda}^2}\varphi^+_j(|\nabla v_j^\lambda|)\,\d y\\
&\le \gamma_j\rho_j\int_{(E_j^\lambda\setminus A_h)\cap S_{j,\lambda}^2}\max\{(|\nabla{v}_j^\lambda|\sigma_j)^p,(|\nabla{v}_j^\lambda|\sigma_j)^q\}\varphi_{\rho_j}^+(1)\,\d y\\
&\le \kappa_d {(Lq\omega(\rho_j))^{\frac{p}{q}}}\frac{L}{p}\gamma_j\rho_j \underset {j\to+\infty}{\longrightarrow} 0.
\end{split}
\]
In {$S_{j,\lambda}^1$ condition \vauno{} holds,}  then
\[
\begin{split}
\int_{(E_j^\lambda\setminus A_h)\cap S_{j,\lambda}^1}\varphi^+_j(|\nabla v_j^\lambda|)\,\d y&\le 2\int_{(E_j^\lambda\setminus A_h)\cap S_{j,\lambda}^1}\varphi^-_j(|\nabla v_j^\lambda|)\,\d y\le c\,\varphi_j^-(\lambda)\,\mathcal{L}^d(E_j^\lambda\setminus A_h)\\
&\le c\frac{\varphi_j^-(\lambda)}{\lambda}\mathcal{H}^{d-1}(S_{v_j})+c\int_{\{M|\nabla{\hat v}_j|>\lambda\}\setminus A_h}\varphi_j^-(M|\nabla {\hat v}_j|)\,\d y\\
&\le c\lambda^{q-1}\mathcal{H}^{d-1}(S_{v_j})+c\int_{\{M|\nabla{\hat v}_j|>\lambda\}\setminus A_h}\varphi_j^-(M|\nabla {\hat v}_j|)\,\d y,
\end{split}
\]
where we used {\eqref{eq:liptrunc},} \eqref{ejlambda}, \eqref{eq:stimevarphij}. From the equiintegrability of the functions $\varphi_j^-(M|\nabla {\hat v}_j|)$ in ${B_1\setminus A_h}$ and from \eqref{misura}, given $\eta>0$, we fix $\lambda=\lambda(\eta)$ sufficiently large in order that 
\begin{equation}\label{piccolo}
c\int_{\{M|\nabla{\hat v}_j|>\lambda\}\setminus A_h}\varphi_j^-(M|\nabla {\hat v}_j|)\,\d y<\eta.
\end{equation}
Therefore we can state that 
\begin{equation*}
\limsup_{j\to+\infty}\int_{E_j^\lambda\setminus A_h}\varphi_j(y,|\nabla v_j^\lambda|)\,\d y< \eta.
\end{equation*}
Concerning the first term in \eqref{liminf}, for the above fixed $\lambda=\lambda(\eta)$, the sequence $(v_j^\lambda)_j$ is equibounded in $W^{1,\infty}(B_1)$, therefore, up to a subsequence, it converges to a function $v^\lambda$ weakly$^*$ in $W^{1,\infty}(B_1)$ and in measure. Moreover, by the lower semicontinuity under convergence in measure of the map
\[
w\mapsto\mathcal{L}^d(\{x\in B_1\setminus A_h : w(x)\neq 0\}),
\]
then
\begin{equation}\label{misura2}
\begin{split}
\lambda^p\mathcal{L}^d(\{x\in B_1\setminus A_h : v^\lambda\neq v_0\})&\le\liminf_{j\to+\infty}\lambda^p\mathcal{L}^d(\{x\in B_1\setminus A_h : v_j^\lambda\neq {\hat v}_j\})\\
&\le \liminf_{j\to+\infty}\lambda^p\mathcal{L}^d(E_j^\lambda\setminus A_h)\\
&\le \liminf_{j\to+\infty} \frac{\lambda^p}{\varphi_j^-(\lambda)}\int_{\{M|\nabla {\hat v}_j|>\lambda\}\setminus A_h}\varphi_j^-(M |\nabla {\hat v}_j|)\,\d y\\
&\le \liminf_{j\to+\infty}\frac{1}{\varphi_j^-(1)} \int_{\{M|\nabla {\hat v}_j|>\lambda\}\setminus A_h}\varphi_j^-(M |\nabla {\hat v}_j|)\,\d y\\
&\le c \liminf_{j\to+\infty} \int_{\{M|\nabla {\hat v}_j|>\lambda\}\setminus A_h}\varphi_j^-(M |\nabla {\hat v}_j|)\,\d y\le c\,\eta,
\end{split}
\end{equation}
using that $\varphi_j^-$ satisfies \inc{p}, the bound from above of $\varphi_j^-(1)$, obtained in \eqref{a01} and \eqref{a02}, and \eqref{piccolo}. 
All things considered, {setting $C_s:=\{x\in B_1:\,\, |\nabla v_0(x)|\le s\}$,} from \eqref{liminf} we derive 
\[
\begin{split}
\liminf_{j\to+\infty}\int_{B_1}\varphi_j(y,|\nabla {\hat v}_j|)\,\d y &\ge \int_{B_1\setminus A_h}\varphi_\infty(|\nabla v^\lambda|)\,\d y-\eta\\
&\ge \int_{(B_1\setminus A_h)\cap\{v^\lambda=v_0\}\cap C_s}\varphi_\infty(|\nabla v_0|)\,\d y-\eta \\
&=\int_{(B_1\setminus A_h)\cap C_s}\varphi_\infty(|\nabla v_0|)\,\d y -\int_{(B_1\setminus A_h)\cap\{v^\lambda\neq v_0\}\cap C_s}\varphi_\infty(|\nabla v_0|)\,\d y-\eta \\
& \ge \int_{(B_1\setminus A_h)\cap C_s}\varphi_\infty(|\nabla v_0|)\,\d y - \varphi_\infty(s)\mathcal{L}^d(\{x\in B_1\setminus A_h : v^\lambda\neq v_0\})-\eta \\
&\ge \int_{(B_1\setminus A_h)\cap C_s}\varphi_\infty(|\nabla v_0|)\,\d y-\varphi_\infty(s) c\,\eta-\eta,
\end{split}
\]
where we used \eqref{misura2} in the last inequality. Thus, letting first $\eta$ tend to zero, then $h$ and finally $s$ tend to infinity, we proved \eqref{lsc}. In the same vein, \eqref{lsc} holds in every ball $B_\rho$, for $\rho\in (0,1]$, of course. \\
\noindent
{\it Step 5: Asymptotics.} Now we integrate  ${\mathcal H}^{d-1}(\{{\tilde{\hat v}}_j\neq\tilde{v}_j\}\cap\partial B_\rho)$ with respect to $\rho$, and using coarea formula, \eqref{limite}, and \eqref{convergetronc}(iv), we obtain
\[
\gamma_j\int_0^1{\mathcal H}^{d-1}(\{{\tilde{\hat v}}_j\neq\tilde{v}_j\}\cap\partial B_\rho)\,\mathrm{d}\rho=\gamma_j\,{\mathcal L}^d( \{{\hat v}_j\neq v_j\}\cap B_1 )\underset {j\to+\infty}{\longrightarrow} 0.
\]
Then, up to a subsequence, we may assume that, for almost every $\rho\in (0,1)$,
\begin{equation}\label{limit}
\lim_{j\to+\infty}\gamma_j\,{\mathcal H}^{d-1}(\{{\tilde{\hat v}}_j\neq\tilde{v}_j\}\cap\partial B_\rho)=0.
\end{equation}
Since for any $j$ and for ${\mathcal L}^1$ -a.e. $\rho\in (0,1)$, ${\mathcal H} ^{d-1}(S_{{\hat v}_j}\cap\partial B_\rho)=0$, we can apply {a straightforward adaptation  of \cite[Lemma 7.3]{Ambrosio-Fusco-Pallara:2000} to our setting }
which gives
\begin{equation}\label{eq:7.3h}
F_j( v_j,\gamma_j, B_\rho) \le  F_j({\hat v}_j,\gamma_j,B_\rho)  + \gamma_j\mathcal{H}^{d-1}(\{{\tilde{\hat v}}_j\neq{\tilde v_j}\}\cap\partial B_\rho) +{\rm Dev}_j(v_j,\gamma_j,B_1)\,,
\end{equation}
and
\begin{equation}\label{eq:7.4h}
\begin{split}
{\rm Dev}_j({\hat v}_j,\gamma_j, B_\rho)&\le F_j({\hat v}_j,\gamma_j,B_\rho)-F_j(v_j,\gamma_j,B_\rho)+\gamma_j\mathcal{H}^{d-1}(\{{\tilde{\hat v}}_j\neq{\tilde v_j}\}\cap\partial B_\rho)\\
&+{\rm Dev}_j(v_j,\gamma_j,B_{1})\,.
\end{split}
\end{equation}
Moreover, taking into account that ${\hat v}_j$ is a truncation of $v_j$ and that $\varphi_j(\cdot,0)=0$ we also have 
\begin{equation}\label{eq:7.5h}
F_j({\hat v}_j,\gamma_j,B_\rho)\le F_j(v_j,\gamma_j,B_\rho).
\end{equation}
Thus, if we set for all $\rho<1$,
\[
\alpha(\rho):=\lim_{j\to+\infty}F_j(v_j,\gamma_j,B_\rho)\,,
\]
which exists, up to a subsequence (not relabeled), by virtue of Helly's Selection Theorem since the function $\rho\mapsto F_j(v_j,\gamma_j,B_\rho)$ is increasing and equibounded, thanks to \eqref{eq:7.3h}, \eqref{eq:7.5h}, \eqref{limit}, and \eqref{assv1}, we may conclude that for ${\mathcal L}^1$-a.e. $\rho\in(0,1)$,
\begin{equation}\label{anche}
\alpha(\rho)=\lim_{j\to+\infty}F_j({\hat v}_j,\gamma_j,B_\rho).
\end{equation}
From this, {the second inequality in \eqref{assv1}, \eqref{limit}} and \eqref{eq:7.4h} we also have that
\begin{equation}
\lim_{j\to+\infty}{\rm Dev}_j({\hat v}_j,\gamma_j,B_\rho)=0\,.
\label{anche2}
\end{equation}
{Now, the sequence of Radon measures 
\[
\mu_j:=\varphi_j(\cdot,|\nabla {\hat v}_j|){\mathcal L}^d+\gamma_j{\mathcal H}^{d-1}\res{S_{{\hat v}_j}}
\]
is equibounded in mass on $B_1$ in view of \eqref{assv1}, so we can find a Radon measure $\mu$ on $B_1$ such that
\begin{equation}
\mu_j\rightharpoonup^*\mu \quad \mbox{ on $B_1$}
\label{eq:weakstarconv}
\end{equation}
up to a subsequence (not relabeled). } \\
\\
{\it Step 6: Final comparison and conclusion.}  To derive the desired contradiction we let  $v\in W^{1,\varphi_\infty}(B_1)$ be such that $\{v\neq v_0\}\ssubset B_1$. {We also consider a sequence $(v^\varepsilon)_{\varepsilon>0}\subset W^{1,\infty}(B_1)$ of regularizations of $v$, strongly converging to $v$  in $W^{1,{\varphi_\infty}}(B_1)$ as $\varepsilon\to0$ (which exists since $\varphi_\infty$ satisfies \dec{q}, see, e.g., \cite[Lemma 6.4.5]{HH})}. 

Let $\rho<\rho'\in (0, 1)$, with $\rho'$ such that \eqref{anche} holds, $\mu(\partial B_{\rho'})=\mu(\partial B_\rho)=0$ and $\{v\neq u_0\}\ssubset B_{\rho}$. Let $\phi\in C^\infty_c(B_{\rho'})$ be such that $\phi = 1$ on $B_\rho$, { $0\leq \phi\leq1$, $|\nabla \phi|\leq \frac{2}{\rho'-\rho}$,} and define
$\zeta_j= \phi v^\varepsilon+ (1-\phi) {\hat v}_j$; since $\{\zeta_j\neq {\hat v}_j\}\ssubset B_{\rho'}$, straightforward computations lead to
\begin{equation}\label{compar}
\begin{split}
F_j({\hat v}_j,\gamma_j,B_{\rho'}) & \le F_j(\zeta_j,\gamma_j,B_{\rho'})+{\rm Dev}_j({\hat v}_j,\gamma_j,B_{\rho'})\\
&\le F_j(v^\varepsilon,\gamma_j,B_\rho)+c\left[\int_{B_{\rho'}\setminus B_\rho}\left(\varphi_j(y,|\nabla v^\varepsilon|)+ 
\varphi_j\left(y,\frac{|v^\varepsilon-{\hat v}_j|}{\rho'-\rho}\right)\right)\mathrm{d}y\right]\\
&\,\,\,\,\,\, +c\,\mu_j(B_{\rho'}\setminus B_\rho) 
+{\rm Dev}_j({\hat v}_j,\gamma_j,B_{\rho'})\,,
\end{split}
\end{equation}
for a suitable constant $c\ge 1$ depending only on $L$ and $p, q$. Now we deal with the convergence of the terms inside the square bracket. Using the uniform convergence \eqref{eq:unifconvphij} 
we have that 
\[
\lim_{j\to+\infty}\int_{B_{\rho'}\setminus B_\rho}\varphi_j(y,|\nabla v^\varepsilon|)\,\d y=\int_{B_{\rho'}\setminus B_\rho}\varphi_\infty(|\nabla v^\varepsilon|)\,\d y,
\]
since $|\nabla v^\e|$ is bounded. {As for the other term, we first notice that  $\varphi_j(\cdot,t)\leq\varphi_j^+(t)$, the boundedness of $v^\e$ and \eqref{convergetronc}(ii) entail the equi-integrability of $\left\{\varphi_j\left(\cdot,\frac{|v^\varepsilon-{\hat v}_j|}{\rho'-\rho}\right)\right\}_{j\in\N}$. }
Furthermore, taking into account the  pointwise convergence of  $\varphi_j\left(y,\frac{|v^\varepsilon-{\hat v}_j|}{\rho'-\rho}\right)$ to $\varphi_\infty\left(\frac{|v^\varepsilon-{v}_0|}{\rho'-\rho}\right)$ {implied by \eqref{eq:unifconvphij}} 
we may appeal to Vitali convergence theorem, which ensures that
\[
\lim_{j\to+\infty}\int_{B_{\rho'}\setminus B_\rho}\varphi_j\left(y,\frac{|v^\varepsilon-{\hat v}_j|}{\rho'-\rho}\right)\,\d y=\int_{B_{\rho'}\setminus B_\rho}\varphi_\infty\left(\frac{|v^\varepsilon-{v}_0|}{\rho'-\rho}\right)\,\d y\,.
\]
Therefore, passing to the limit as $j\to+\infty$ in \eqref{compar}, {taking into account \eqref{anche2} and the convergence \eqref{eq:weakstarconv}, {together with the choice of $\rho$ and $\rho'$,}} we have
\[
\alpha(\rho')\le \int_{B_\rho}\varphi_\infty(|\nabla v^\varepsilon|)\,\mathrm{d}y+c\left[\int_{B_{\rho'}\setminus B_\rho}\left(\varphi_\infty(|\nabla v^\varepsilon|)+\varphi_\infty\left(\frac{|v^\varepsilon-v_0|}{\rho'-\rho}\right)\right)\mathrm{d}y\right]+c\,\mu(B_{\rho'}\setminus B_\rho)\,.
\]
Now we  let $\varepsilon\to0$ and recalling that $v=v_0$ outside $B_\rho$, we easily obtain 
\[
\alpha(\rho)\le \alpha(\rho')\le \int_{B_\rho}\varphi_\infty(|\nabla v|)\,\mathrm{d}y+c\int_{B_{\rho'}\setminus B_\rho}\varphi_\infty(|\nabla v|)\mathrm{d}y+c\,\mu(B_{\rho'}\setminus B_\rho)\,.
\]
Therefore, letting $\rho'$ tend to $\rho$ we finally get that for ${\mathcal L}^1$ -a.e. $\rho$ and any $v\in W^{1,{\varphi_\infty}} (B_1)$ such that $\{v\neq v_0\}\ssubset B_\rho$ we have 
\[
\lim_{j\to+\infty}F_j({\hat v}_j,\gamma_j,B_\rho)=\lim_{j\to+\infty}F_j({v}_j,\gamma_j,B_\rho)\le \int_{B_\rho}\varphi_\infty(|\nabla v|)\,\mathrm{d}y\,.
\]
{In particular, the previous inequality holds  for $v=v_0$, whence taking into account the lower semicontinuity result \eqref{lsc} previously obtained in Step~4, we get that} 
\[
\lim_{j\to+\infty}F_j({v}_j,\gamma_j,B_\rho)= \int_{B_\rho}\varphi_\infty(|\nabla v_0|)\,\mathrm{d}y
\]
and that $v_0$ is a local minimizer of the functional $v\mapsto\int_{B_1}\varphi_\infty(|\nabla v|)\,\mathrm{d}x$. Thanks to Proposition \ref{prop:nondip} and Remark \ref{senzac2},  $v_0$ satisfies
\[
\sup_{B_\tau}\varphi_\infty(|\nabla v_0|)\le \sup_{B_{1/2}}\varphi_\infty(|\nabla v_0|)\le C_0\-int_{B_1}\varphi_\infty(|\nabla v_0|)\,\mathrm{d}y.
\]
In conclusion 
\[
\begin{split}
\lim_{j\to+\infty}F_j(v_j,\gamma_j,B_\tau)&=\int_{B_\tau}\varphi_\infty(|\nabla v_0|)\,\mathrm{d}y\le \sup_{B_\tau}\varphi_\infty(|\nabla v_0|)\tau^d \mathcal{L}^d(B_1)\\
&\le C_0\tau^d\int_{B_1}\varphi_\infty(|\nabla v_0|)\,\mathrm{d}y\le C_0\tau^d\limsup_{j\to+\infty}F_j(v_j,\gamma_j,B_1)\,,
\end{split}
\]
which provides the contradiction to \eqref{assv2}.

\end{proof}

\subsection{Ahlfors-type regularity. Strong minimizers.} \label{sec:densitylowerb}

In order to study the regularity of the jump set $S_u$, a key tool will be an Ahlfors-type regularity result, ensuring that $F(u,B_\rho(x))$, where $B_\rho(x)$ is any ball centred at a jump point $x\in S_u$, is controlled from above and from below (see \eqref{eq:7.24AFP}). 

We first recall the definition of \emph{quasi-minimizer} (see \cite[Definition~7.17]{Ambrosio-Fusco-Pallara:2000}):
a function $u\in SBV_{\rm loc}(\Omega)$ is a quasi-minimizer of the functional $F(v,\Omega)$ if there exists a constant $\eta\geq0$ such that for all balls $B_\rho(x)\ssubset\Omega$ it holds that
\begin{equation}
{\rm Dev}(u, B_\rho(x)) \leq \eta \rho^d\,.
\label{eq:devquasim}
\end{equation}
{The class of quasi-minimizers complying with \eqref{eq:devquasim} is denoted by $\mathcal{M}_\eta(\Omega)$.}


{The upper bound in \eqref{eq:7.24AFP} follows from a standard comparison argument, and here the assumption $x\in S_u$ is, actually, not needed. }
{On the contrary, the lower bound for $F(u,B_\rho(x))$ therein requires that the small balls $B_\rho(x)$ be centred at $x\in {S_u}$.} The proof is based on the decay estimate of Lemma~\ref{lem:decay}, and it follows along the lines of the proof of \cite[Theorem~7.21]{Ambrosio-Fusco-Pallara:2000} where $p$ is constant, {or that of \cite[Theorem~4.7]{LSSV}. We then only sketch the proof, just indicating the main changes. } 

\begin{theorem}[Ahlfors-type regularity]\label{thm:mainthm} {Let $\varphi$ be a function satisfying {\rm\ref{hpuno}}, {\rm\ref{hpdue}}, {\rm\ref{hptre}}.} 
There exist $\theta_0$ and $\rho_0$ depending on $d, p, q$ with the property that if $u\in SBV(\Omega)$ is a {quasi-minimizer} of $F$ in $\Omega$, $u\in \mathcal{M}_\eta(\Omega)$, then 
\begin{equation}
{\theta_0\rho^{d-1}<F(u, B_\rho(x))\leq d\kappa_d\rho^{d-1} + {\eta\rho^d} }
\label{eq:7.24AFP}
\end{equation}
for all balls $B_\rho(x)\subset\Omega$ with centre $x\in\overline{S}_u$ and radius $\rho<\frac{\rho_0}{\eta}$. 
Moreover, 
\begin{equation}
\mathcal{H}^{d-1}((\overline{S}_u\setminus S_u)\cap \Omega)=0.
\label{eq:7.24AFPbis}
\end{equation}
\end{theorem}
\begin{proof}
{For the proof of the upper bound in \eqref{eq:7.24AFP} see, e.g., \cite[Lemma~7.19]{Ambrosio-Fusco-Pallara:2000}.} As for the lower bound, without loss of generality, we may assume that $x=0$. Let $0<\tau<1$ be fixed such that $\sqrt{\tau}\leq \frac{1}{C_1}$ and set $\varepsilon_0:=\varepsilon(\tau)$, where $C_1$ and $\varepsilon(\tau)$ are given from Lemma~\ref{lem:decay}. Let $0<\sigma,\rho_0<1$ be chosen such that
\begin{equation*}
{\sigma \leq \frac{\varepsilon_0}{C_1 (d\kappa_d+1)} \quad \mbox{ and }\quad \rho_0:= \min\{1,\varepsilon(\sigma)^2,\varepsilon_0\tau^d\theta(\tau), \varepsilon_0\sigma^{d-1}\theta(\sigma)\}\,, }
\end{equation*}
where $\theta(\tau)$ and $\theta(\sigma)$ are the constants of Lemma~\ref{lem:decay} corresponding to $\tau$ and $\sigma$, respectively.
{As a first step, an argument by induction as for \cite[(4.23)-(4.24)]{LSSV}  shows that if $\rho<\frac{\rho_0}{\eta}$ and $B_\rho\subset\Omega$, then
\begin{equation}
F(u, B_\rho) \leq \varepsilon(\sigma)\rho^{d-1}
\label{eq:7.25AFP}
\end{equation}
implies
\begin{equation}
F(u, B_{\sigma\tau^m\rho}) \leq \varepsilon_0 \tau^\frac{m}{2}(\sigma\tau^m\rho)^{d-1}\,,\quad \forall m\in\N \,.
\label{eq:7.26AFP}
\end{equation}}
Now, assuming \eqref{eq:7.25AFP} for some ball $B_\rho\subset\Omega$, with $\rho<\frac{\rho_0}{\eta}$, from \eqref{eq:7.26AFP} we get
\begin{equation*}
\lim_{r\to0} \frac{F(u, B_r)}{r^{d-1}}=0\,,
\end{equation*}
whence, {by using the inequality $\frac{t^{p}}{Lq} \leq 1+ \varphi^-_{B_r}(t)$, we infer
\begin{equation*}
\lim_{r\to0} \frac{1}{r^{d-1}} \int_{B_r(x)}|\nabla u|^{p}\,\mathrm{d}y=0\,.
\end{equation*}} Therefore, {Theorem~\ref{thm:thm7.8AFP} with $p$ above} and $q=1^*$ implies that $0\in I$, where
\[
I :=\left\{x\in\Omega : \limsup_{r\to 0}\-int_{B_r(x)}|u(y)|^{1^*}\,\mathrm{d}y=+\infty\right\}.
\]
Then \eqref{eq:7.24AFP} holds true for all $x\in S_u\setminus I$, {by a density argument, the inequality is still true for balls centred in  $x\in \overline{S_u\setminus I}$. } 
We are left to prove that $\overline{S_u\setminus I}=\overline{S}_u$.
Let $x\not\in\overline{S_u\setminus I}$; we first observe that the set $I$ is $\mathcal{H}^{d-1}$-negligible (see \cite[Lemma 3.75]{Ambrosio-Fusco-Pallara:2000}. Thus, we can find a neighborhood $V$ of $x$ such that $\mathcal{H}^{d-1}(V\cap S_u)=0$, and this implies in a standard way that $u\in W^{1,p}(V)$. {Now, by virtue of the classical Poincar\'e inequality for Sobolev functions, the upper bound in \eqref{eq:7.24AFP} 
and the Campanato's characterization of H\"older continuity (see, e.g., \cite{Camp1}), we infer that a representative of $u$ belongs to $C^{0,\alpha}(V)$, where $\alpha:=\frac{p-1}{p}$. Therefore, $x\not\in \overline{S}_u$, and this concludes the proof of \eqref{eq:7.24AFP}. }

As for \eqref{eq:7.24AFPbis}, it follows from \eqref{eq:7.24AFP} by a geometric measure theory argument. Let us define the set 
\begin{equation*}
\Sigma:=\left\{x\in\Omega:\,\,\mathop{\lim\sup}_{r\to0} \frac{1}{r^{d-1}}\int_{B_r(x)}\varphi(y,|\nabla u|)\,\mathrm{d}y>0\right\}\,.
\label{eq:setsigma}
\end{equation*}
{Since $\varphi(\cdot,|\nabla u|)\in L^1_{\rm{loc}}(\Omega)$,} it holds that $\mathcal{H}^{d-1}(\Sigma)=0$  (see, e.g., \cite[Section~2.4.3, Theorem~3]{EG}). {The proof of \eqref{eq:7.24AFPbis} then goes exactly as for the proof of \cite[eq. (4.20)]{LSSV}, we then omit the details.}

\end{proof}

{We are now in position to prove an existence result for minimizers of free-discontinuity functionals of the form
\begin{equation}
\mathcal{F}(u):= \int_\Omega \varphi(x,|\nabla u|)\,\mathrm{d}x + \alpha\int_\Omega|u-g|^q\,\mathrm{d}x +  \mathcal{H}^{d-1}(S_u\cap\Omega)\,, 
\label{eq:functionals}
\end{equation}
defined for $u\in SBV^{\varphi}(\Omega)$, where $\alpha>0$, $q\geq1$, and $g\in L^\infty(\Omega)$, which are the weak formulation of
\begin{equation}
\mathcal{G}(K,u):= \int_{\Omega\setminus K} \varphi(x,|\nabla u|)\,\mathrm{d}x + \alpha\int_{\Omega\setminus K}|u-g|^q\,\mathrm{d}x +  \mathcal{H}^{d-1}(K\cap\Omega)
\label{eq:functionalsG}
\end{equation}
where $u\in W^{1,\varphi}(\Omega\setminus K)$ and $K\subset\R^d$ is a closed set.}


{First, we notice that in order to minimize $\mathcal{F}$, we may restrict to those $u\in SBV(\Omega)$ such that $\|u\|_\infty\leq \|g\|_\infty<\infty$. Indeed, since the integrand $\varphi$ does not depend on $u$ and $t\mapsto\varphi (\cdot,t)$ is non-decreasing, it is immediate to check that $\mathcal{F}(u)$ is non increasing by truncations.} 

\begin{theorem}\label{thm:main}
Let $\varphi$ comply with the assumptions of Theorem~\ref{thm:mainthm}. Then there exists a minimizer $u\in SBV(\Omega)\cap L^\infty(\Omega)$ of $\mathcal{F}$ defined in \eqref{eq:functionals}. Moreover, the pair $(\overline{S}_u,u)$ is a (strong) minimizer of the functional $\mathcal{G}$ \eqref{eq:functionalsG}.
\end{theorem}

\begin{proof}
{As $\varphi(x, \cdot)$ is convex and superlinear for every $x\in \Omega$, the existence of a bounded minimizer $u$ is based on a nowadays classical argument, combining the De Giorgi's lower semicontinuity theorem \cite{DeG} with the closure and compactness results in $SBV$ by Ambrosio (see \cite[Theorem~4.7 and 4.8]{Ambrosio-Fusco-Pallara:2000}). Moreover, it is easy to check (see, e.g., \cite[Remark~7.16]{Ambrosio-Fusco-Pallara:2000}) that any minimizer $u$ of the functional $\mathcal{F}$  in \eqref{eq:functionals}  belongs to $\mathcal{M}_\eta(\Omega)$ with $\eta:= 2^q \alpha \kappa_d \|g\|_\infty^q$, thus \eqref{eq:7.24AFPbis} holds. The rest of the proof closely follows the argument of, e.g., \cite[Theorem 4.8]{LSSV}, so we omit further details. } 
\end{proof}

\section*{Acknowledgements} 

The authors are members of Gruppo Nazionale per l'Analisi Matematica, la Probabilit\`a e le loro Applicazioni (GNAMPA) of INdAM.
G. Scilla and F. Solombrino have been supported by the project STAR PLUS 2020 – Linea 1 (21‐UNINA‐EPIG‐172) ``New perspectives in the Variational modeling of Continuum Mechanics''. The work of F. Solombrino is part of the project “Variational Analysis of Complex Systems in Materials Science, Physics and Biology” PRIN Project 2022HKBF5C. {The research of C. Leone and A. Verde was supported by PRIN Project 2022E9CF89 ``Geometric Evolution Problems and Shape Optimizations''.  PRIN projects are part of PNRR Italia Domani, financed by European Union through NextGenerationEU. }

 \bibliographystyle{siam}

\end{document}